\documentclass[a4paper]{amsart}

\usepackage{amsmath}
\usepackage{t1enc} 
\usepackage{epsf}
\usepackage{amssymb}
\usepackage{graphicx}
\usepackage{bbm}

\newtheorem{theorem}{Theorem}[section]
\newtheorem{lem}[theorem]{Lemma}
\newtheorem{prop}[theorem]{Proposition}

\newtheorem{coro}[theorem]{Corollary}
\newtheorem{fact}[theorem]{Fact}

\theoremstyle{definition}
\newtheorem{defi}[theorem]{Definition}
\newtheorem{ex}[theorem]{Example}
\newtheorem{rem}[theorem]{Remark}

\numberwithin{equation}{section}
\numberwithin{theorem}{section}

\newcommand{\Z}{\mathbb{Z}}
\newcommand{\N}{\mathbb{N}}
\newcommand{\Q}{\mathbb{Q}}
\newcommand{\R}{\mathbb{R}}
\newcommand{\C}{\mathbb{C}}

\begin{document}

\title[Uniqueness in discrete tomography of algebraic Delone sets]{Solution of a uniqueness
  problem in the discrete tomography of algebraic Delone sets}
\author{Christian Huck}
\author{Michael Spie\ss}
\address{Fakult\"{a}t f\"{u}r Mathematik,
  Universit\"{a}t Bielefeld, Postfach 100131, 33501 Bielefeld, Germany}
\email{huck@math.uni-bielefeld.de}
\email{mspiess@math.uni-bielefeld.de}


\begin{abstract}
We consider algebraic Delone sets $\varLambda$ in the
Euclidean plane and address the problem of
distinguishing convex subsets of $\varLambda$ by X-rays in prescribed $\varLambda$-directions, i.e., directions parallel to nonzero
interpoint vectors of $\varLambda$.  Here, an X-ray in direction
$u$ of a finite
 set gives the number of points in the set on each line
parallel to $u$. It is shown
that for any algebraic Delone set $\varLambda$ there are four
prescribed $\varLambda$-directions such that any two convex subsets of
$\varLambda$ can be distinguished by the corresponding X-rays. We
further prove the existence of a natural number
$c_{\varLambda}$ such that any two convex subsets of $\varLambda$ can be distinguished by their X-rays in any set of
$c_{\varLambda}$ prescribed
$\varLambda$-directions. In particular, this extends a well-known result of Gardner
and Gritzmann on the corresponding problem for planar lattices to
nonperiodic cases that are relevant in quasicrystallography.
\end{abstract}

\maketitle

\section{Introduction}\label{intro}
{\em Discrete tomography} is concerned with the 
inverse problem of retrieving information about some {\em finite}
object in Euclidean space from 
(generally noisy) information about its slices. One important problem
is the {\em unique reconstruction} of a finite point set in Euclidean $3$-space 
from its {\em (discrete parallel) X-rays} in a small number of directions, where the
  {\em X-ray} of the finite set in a certain direction is the {\em line sum
  function} giving the
number of points in the set on each line parallel to this direction. 

The interest in the discrete tomography of planar Delone sets $\varLambda$
with long-range order is motivated by the requirement in materials
science for the unique reconstruction of solid state materials like {\em quasicrystals}
slice by slice from their images under quantitative {\em high
  resolution transmission electron microscopy} (HRTEM). In fact,
in~\cite{ks}, \cite{sk} a technique
is described, which can, for certain 
crystals, effectively measure the number of atoms lying on densely occupied 
columns. It is reasonable to expect that future developments in
technology will extend this situation to other solid state
materials. The aforementioned density condition forces us to consider only $\varLambda$-directions, i.e.,
directions parallel to nonzero interpoint vectors of $\varLambda$. Further, since typical objects may be damaged or even destroyed by the radiation
energy after about $3$ to $5$ images taken by
HRTEM, applicable results may only use a small 
number of X-rays. It actually is this restriction to few high-density
directions that makes the
problems of discrete tomography mathematically challenging, even if
one assumes the absence of noise. 

In the
 traditional setting, motivated by {\em crystals}, the positions to be
 determined form a finite subset of a 
 three-dimensional lattice, the latter allowing a slicing into equally
 spaced 
 congruent copies of a planar lattice.  In the crystallographic
 setting, by the affine nature
 of the problem, it therefore suffices to study
 the discrete tomography of the square lattice;
 cf.~\cite{GG}, \cite{GG2}, \cite{GGP}, \cite{Gr}, \cite{HT}, \cite{HK},
 \cite{HK2} for an
 overview. For the 
 {\em quasicrystallographic} setting, the
 positions to be determined form a finite subset of a {\em nonperiodic Delone
   set} with {\em long-range order} 
 (more precisely, a {\em mathematical
   quasicrystal}  or {\em model set}~\cite{BM}, \cite{Moody}) which on the other hand is contained
 in a free additive subgroup of $\R^3$ of finite rank $r> 3$. These model sets
 possess, as it is the case for lattices, a dimensional hierarchy, i.e., they allow a slicing into
 planar model sets. However, the slices are in general no longer pairwise congruent or equally
 spaced in $3$-space; cf.~\cite{PABP2}.  Still, most of
 the model sets that describe real quasicrystallographic structures
 allow a slicing such that each slice is, when seen from a common
 perpendicular viewpoint, a (planar) {\em $n$-cyclotomic model set},
 where $n=5$, $n=8$ and $n=12$, respectively (Example~\ref{algex});
 cf.~\cite[Sec.~1.2]{H}, \cite{H2}, \cite[Sec.~4.5]{H5} and \cite{St}
 for details. These cyclotomic
 model sets thus take over the role played by the planar lattices in
 the crystallographic case. In the present text, we shall focus
 on the larger class of {\em algebraic Delone sets} (Definition~\ref{algdeldef}).

Since different finite subsets of a Delone set $\varLambda$ may have the same X-rays in several
$\varLambda$-directions (in other words, the above problem of uniquely
reconstructing a finite point set from its X-rays is an
{\em ill-posed}\/ problem in general), one is naturally interested in conditions
to be imposed
on the set of $\varLambda$-directions together with restrictions on
the possible finite subsets of
$\varLambda$ such that the latter phenomenon
cannot occur. Here, we consider the {\em convex subsets}
of $\varLambda$ (i.e., bounded subsets of $\varLambda$ with
the property that their convex hull contains no new points of $\varLambda$) and show that for any algebraic Delone
set $\varLambda$ there are four prescribed $\varLambda$-directions such that any two convex subsets of
$\varLambda$ can be distinguished by the corresponding X-rays,
whereas less than four $\varLambda$-directions never suffice for this
purpose (Theorem~\ref{dtmain}(a)). We
further prove the existence of a finite number
$c_{\varLambda}$ such that any two convex subsets of $\varLambda$
can be distinguished by their X-rays in {\em any} set of
$c_{\varLambda}$ prescribed
$\varLambda$-directions (Theorem~\ref{dtmain}(b)). Moreover, we
demonstrate that the
least possible numbers $c_{\varLambda}$ in the
case of the practically most relevant examples of $n$-cyclotomic model
sets $\varLambda$ with $n=5$, $n=8$ and $n=12$ are (in that
 very order) $11$, $9$ and $13$ (Theorem~\ref{dtmain2}(b) and Remark~\ref{rembest}). This extends a well-known result of Gardner
and Gritzmann (cf.~\cite[Thm.~5.7]{GG}) on the corresponding
problem for planar lattices $\varLambda$ ($c_{\varLambda}=7$) to
cases that are relevant in quasicrystallography and particularly solves Problem
4.34 of~\cite{H5}. The above results and their continuous analogue 
(Theorem~\ref{tmain})  follow from deep insights into the existence
of certain {\em $U$-polygons} in the plane (cf.~Sec.~\ref{sec1}). We believe that our main
result on these polygons (Theorem~\ref{main}) is of independent
interest from a purely geometrical point of view.  For the algorithmic 
reconstruction problem in the quasicrystallographic setting, we refer
the reader to~\cite{BG2},~\cite{H2}.

\section{Preliminaries and notation}\label{sec1}

Natural numbers are always assumed to be positive. We denote the norm in Euclidean $d$-space by $\Arrowvert \cdot \Arrowvert$. The Euclidean plane will occasionally be
identified with the complex numbers. For $z\in\C$, $\bar{z}$ denotes the 
complex conjugate of $z$ and $|z|=\sqrt{z\bar{z}}$ its modulus. The unit circle in $\C$ is
denoted by $\mathbb{S}^{1}$ and its elements are also called
{\em directions}. For $z\in \C^*$, we denote
by $\operatorname{sl}(z)$ the slope of $z$, i.e., $\operatorname{sl}(z)=-i(z-\bar{z})/(z+\bar{z})\in
\mathbb{R}\cup\{\infty\}$. For $r>0$ and $z\in\C$,
$B_{r}(z)$ is the open ball of radius $r$ about $z$. Recall that an
($\R$-){\em linear endomorphism} (resp., {\em affine endomorphism})
 of $\C$ is given
by $z \mapsto az+b\bar{z}$ (resp., $z \mapsto az+b\bar{z}+t$), where $a,b,t\in\C$. In both cases,
it is an automorphism if and only if $az+b\bar{z}=0$ only holds for $z=0$. A {\em homothety} $h\!:\, \C \rightarrow
\C$ is given by $z \mapsto \lambda z + t$, where
$\lambda \in \R$ is positive and $t\in \C$. In the following, let
$\varLambda$ be a subset of $\C$. A direction
$u\in\mathbb{S}^{1}$ is called a $\varLambda${\em-direction} if it is
parallel to a nonzero element of the difference set
$\varLambda-\varLambda=\{v-w\,|\,v,w\in\varLambda\}$ of $\varLambda$. A {\em convex polygon} is the convex hull of a finite set of points in $\C$. A {\em polygon in} $\varLambda$ is a convex polygon with all vertices in $\varLambda$. Further, a bounded subset $C$ of $\varLambda$ is called a {\em convex subset of} $\varLambda$ if $C =
\operatorname{conv}(C)\cap \varLambda$, where $\operatorname{conv}(C)$
denotes the convex hull of $C$. Let $U\subset \mathbb{S}^{1}$ be
a finite set of directions. A nondegenerate convex polygon $P$ is
called a {\em $U$-polygon} if it has the property that whenever $v$ is
a vertex of $P$ and $u\in U$, the line in the complex plane in direction $u$
which passes through $v$ also meets another vertex $v'$ of $P$. By a
{\em regular polygon} we shall always mean a nondegenerate convex regular polygon. An
{\em affinely regular polygon} is the image of a
regular polygon under an affine automorphism of the complex plane. $\varLambda$ is called {\em uniformly discrete} if there is a radius
$r>0$ such that every ball $B_{r}(z)$ with $z\in\C$ contains at most one point of
  $\varLambda$. Note that the bounded 
  subsets of a uniformly discrete set $\varLambda$ are precisely the
  finite subsets of $\varLambda$. $\varLambda$ is called {\em relatively dense} if there is a radius $R>0$
  such that every ball $B_{R}(z)$ with $z\in\C$ contains at least one point of
  $\varLambda$. $\varLambda$ is called a {\em Delone set} if it is both uniformly
  discrete and relatively dense. $\varLambda$ is said to be of
  \emph{finite local complexity} if $\varLambda-\varLambda$ is
  discrete and closed. Note that $\varLambda$ is of finite local complexity if and
  only if for every $r>0$ there
  are, up to translation, only finitely many \emph{patches of radius
    $r$}, i.e., sets of the form $\varLambda\cap
  B_{r}(z)$, where $z\in\C$; cf.~\cite{Moody}.  A Delone set $\varLambda$ is a \emph{Meyer set} if $\varLambda-\varLambda$ is
  uniformly discrete. Trivially, any Meyer set
  is of finite local complexity. $\varLambda$ is called {\em periodic} if it has
nonzero translation symmetries. Finally, we denote by $K_{\varLambda}$
the intermediate field of $\C/\Q$ that is given by 
$$
K_{\varLambda}\,\,:=\,\,\Q\left(\left(\varLambda-\varLambda\right)\cup\left(\overline{\varLambda-\varLambda}\right)\right)\,.
$$

\subsection{Recollections from the theory of cyclotomic fields}

Let $K\subset\C$ be a field and let $\mu$ be the group of roots of unity in
$\C$. We denote the maximal real subfield $K\cap\R$ of $K$ by $K^+$ and set $\mu(K):=\mu\cap K$. For $n\in \mathbb{N}$, we always let $\zeta_n := e^{2\pi
    i/n}$, a primitive $n$th root of unity in $\C$. Then, $\Q(\zeta_n)$ is the
  $n$th cyclotomic field.  Further, $\phi$ will always denote Euler's totient function, i.e., $$\phi(n) =
\operatorname{card}\left(\big\{k \in \mathbb{N}\,\big |\,1 \leq k \leq n
  \textnormal{ and } \operatorname{gcd}(k,n)=1\big\}\right)\,.$$
Recall that $\phi$ is multiplicative with $\phi(p^r)=p^{r-1}(p-1)$ for
$p$ prime and $r\in\N$. 

\begin{fact}[Gau\ss]\cite[Thm.~ 2.5]{Wa}\label{gau}
 $[\Q(\zeta_n) :
  \Q] = \phi(n)$ and the field extension $\Q(\zeta_n)/ \Q$
is a Galois extension with Abelian Galois group $G(\Q(\zeta_n)/
\Q) \simeq (\Z / n\Z)^{\times}$,
with $a \pmod n$ corresponding to the automorphism given by $\zeta_n \mapsto \zeta_n^{a}$.\qed
\end{fact}

Note that the composition $\Q(\zeta_n)\Q(\zeta_m)=\Q(\zeta_n,\zeta_m)$
of cyclotomic fields is equal to the cyclotomic field 
$\Q(\zeta_{\operatorname{lcm}(n,m)})$. Further, the intersection
$\Q(\zeta_n)\cap\Q(\zeta_m)$ of cyclotomic fields is equal to the cyclotomic field 
$\Q(\zeta_{\operatorname{gcd}(n,m)})$.  Note that $\Q(\zeta_n)^+=\Q(\zeta_n+\bar{\zeta}_n)=\Q(\zeta_n+\zeta_n^{-1})$. Clearly, if $n$ divides $m$ then
$\Q(\zeta_n)$ is a subfield of $\Q(\zeta_m)$. Since
$\Q(\zeta_n)=\Q(\zeta_{2n})$ for odd $n$ by Fact~\ref{gau}, we may sometimes restrict ourselves to $n\in\N$ with $n\not\equiv 2 \pmod 4$.

\subsection{Cross ratios}

Let $(t_1,t_2,t_3,t_4)$ be an ordered tuple of four pairwise distinct
elements of $\mathbb{R}\cup\{\infty\}$. Then, its {\em cross ratio}
$\langle t_1,t_2,t_3,t_4\rangle$ is the nonzero real number defined by
$$
\langle t_1,t_2,t_3,t_4\rangle := \frac{(t_3 - t_1)(t_4 - t_2)}{(t_3 - t_2)(t_4 - t_1)}\,,
$$
with the usual conventions if one of the $t_i$ equals
$\infty$. We need the following invariance property of cross ratios of slopes.

\begin{fact}\cite[Lemma 2.17]{H}\label{crossratio}
Let $z_1,z_2,z_3,z_4\in \C^*$ be  pairwise
nonparallel and let $\Psi$ be a linear automorphism of
the complex plane. Then, one has
$$
\big\langle \operatorname{sl}(z_{1}),\operatorname{sl}(z_{2}),\operatorname{sl}(z_{3}),\operatorname{sl}(z_{4})\big\rangle = \big\langle \operatorname{sl}(\Psi(z_1)),\operatorname{sl}(\Psi(z_2)),\operatorname{sl}(\Psi(z_3)),\operatorname{sl}(\Psi(z_4))\big\rangle\,.\qed
$$  
\end{fact}

\begin{fact}\cite[Lemma 2.20]{H}\label{crkn4gen}
Let $\varLambda\subset\C$.  Then the
cross ratio of slopes of four pairwise nonparallel
$\varLambda$-directions is an element of $K_{\varLambda}^+$.\qed
\end{fact}

\section{Algebraic Delone sets}

 The following notions will
be useful; see also~\cite{H3}, \cite{H4}, \cite{H5} for
generalisations and for results related to
those presented below.

\begin{defi}\label{algdeldef}
A Delone set $\varLambda\subset\C$ is called an {\em algebraic Delone
  set} if it satisfies the following properties:
\begin{eqnarray*}
\mbox{(Alg)}&&[K_{\varLambda}:\Q]<\infty\,.\\
\mbox{(Hom)}&&\mbox{For any finite subset $F$ of $K_{\varLambda}$, there is a
homothety}\\&&\mbox{$h$ of the complex plane such that $h(F)\subset \varLambda$\,. 
}
\end{eqnarray*}
Moreover, $\varLambda$ is called an {\em $n$-cyclotomic
  Delone set} if it satisfies the property
$$\mbox{($n$-Cyc)}\quad\quad K_{\varLambda}\subset\Q(\zeta_n)$$
for some $n\geq 3$
and has property (Hom).  Further,
$\varLambda$ is called a {\em cyclotomic Delone set} if it is an
$n$-cyclotomic Delone set for a suitable $n\geq 3$.
\end{defi} 

\begin{rem}\label{remalg}
Algebraic Delone
sets were already introduced in~\cite[Definition
4.1]{H5}. Clearly, for every algebraic Delone set $\varLambda$, the
field 
extension 
$K_{\varLambda}/\Q$ is an imaginary extension (due to $\varLambda$
being relatively dense) with $\overline{K_{\varLambda}}=K_{\varLambda}$. By the Kronecker-Weber theorem (cf.~\cite[Thm.~
14.1]{Wa}) and Fact~\ref{gau}, the cyclotomic Delone sets are precisely the algebraic
Delone sets $\varLambda$ with the additional property that
$K_{\varLambda}/\Q$ is an Abelian extension. 
\end{rem}

Following Moody~\cite{Moody}, modified along the lines of the algebraic setting of Pleasants~\cite{PABP}, we
define as follows.

\begin{defi}\label{algmodel}
Let $K\subset\C$ be an imaginary quadratic extension
of a real algebraic 
number field (necessarily, this real algebraic number field is $K^+$) of degree
$[K:\Q]=:d$ over $\Q$  (in particular, $d$ is even). Let
$\mathcal{O}_{K}$ be the ring of integers in $K$ and let $.^{\star}\!:\,
  \mathcal{O}_{K}\rightarrow \C^{s-1}\times\R^t$ be any map of the form $z\mapsto
  (\sigma_{2}(z),\dots,\sigma_{s}(z),\sigma_{s+1}(z),\dots,\sigma_{s+t}(z))$,
  where
$\sigma_{s+1},\dots,\sigma_{s+t}$ are the real embeddings of
$K/\Q$ into $\C/\Q$ and $\sigma_{2},\dots,\sigma_{s}$ arise from the complex embeddings of
$K/\Q$ into $\C/\Q$ except the identity and the complex conjugation by
choosing exactly one embedding 
from each pair of complex conjugate ones (in particular, $d=2s+t$ and
$s\geq 1$). Then, for any such choice,
each translate $\varLambda$ of
$$
\varLambda(W):=\{z\in\mathcal{O}_{K}\,|\,z^{\star}\in W\}\,,
$$  
where $W\subset
  \C^{s-1}\times\R^t\simeq\R^{d-2}$ is a relatively compact set with nonempty interior, is called a \emph{$K$-algebraic model set}. Moreover, $.^{\star}$ and $W$ are called the \emph{star
    map} and the \emph{window} of $\varLambda$, respectively. 
\end{defi}

\begin{rem}\label{modelrem}
Algebraic number fields $K$ as above may be obtained by starting with
a real algebraic number field $L$ and adjoining the square root of a
negative number from $L$. Note that, in the situation of Definition~\ref{algmodel}, the
quadratic extension $K/K^+$ is a Galois extension with $G(K/K^+)$
containing the identity and the complex conjugation (in particular,
one has $\overline{K}=K$). We use the
convention that for $d=2$ (meaning that $s=1$ and $t=0$), $\C^{s-1}\times\R^t$ is the trivial group $\{0\}$ and the star map is the zero
  map. Due to the Minkowski
representation 
$
\{(z,z^{\star})\,|\,z\in\mathcal{O}_{K}\}
$ of the maximal order $\mathcal{O}_{K}$ of $K$   
being a (full) lattice in $\C\times\C^{s-1}\times\R^t\simeq \R^d$ (cf.~\cite[Ch.~2,
Sec.~3]{Bo}) that is in one-to-one correspondence with
$\mathcal{O}_{K}$ via the canonical 
projection on the first factor and due to 
$\mathcal{O}_{K}^{\star}$ being a dense subset of
$\C^{s-1}\times\R^t$ (see Lemma~\ref{dense} below),
$K$-algebraic model sets are indeed {\em model sets} and thus
are 
Meyer sets; cf.~\cite{BM}, \cite{BM2}, \cite{Moody},  \cite{Schl2},  \cite{Schl} for the
  general setting and further properties of model sets. Since the
  star map is a monomorphism of Abelian groups for $d>2$ and since the
  window is a bounded set, a $K$-algebraic model set
  $\varLambda$ is periodic if and only if $d=2$, in which case $\varLambda$ is a
  translate of the planar lattice $\mathcal{O}_K$. 
\end{rem}

A real algebraic integer $\lambda$ is called a {\em Pisot-Vijayaraghavan number} ({\em
  PV-number}\/) if $\lambda>1$ while all other conjugates of
$\lambda$ have moduli strictly less than $1$. 

\begin{fact}\cite[Ch.~1, Thm.~ 2]{Sa}\label{pisot}
Every real algebraic number field contains a primitive element that is
a PV-number.\qed
\end{fact}

Before we can show that $K$-algebraic model sets are
algebraic Delone sets, we need the following lemmata. 

\begin{lem}\label{r2}
Let $\varLambda$ be a nonperiodic $K$-algebraic model set
with star map $.^{\star}$. Then,  there is an algebraic integer
$\lambda\in K^+$ such that a suitable power of the $\Z$-module 
endomorphism $m_{\lambda}^{\star}$ of $\mathcal{O}_K^{\star}$,
defined by $m_{\lambda}^{\star}(z^{\star})=(\lambda z)^{\star}$,  
is contractive, i.e., there is an $l\in\N$ and a real number 
$c \in (0,1)$ such that $\Arrowvert (m_{\lambda}^{\star})^l(z^{\star})\Arrowvert\leq c\, \Arrowvert z^{\star}\Arrowvert$ holds for all $z\in \mathcal{O}_{K}$.
\end{lem}
\begin{proof} 
By Fact~\ref{pisot}, we may choose a
PV-number $\lambda$ of degree $d/2=[K^+:\Q]$ in
$K^+$, where $d=[K:\Q]\geq 4$ due to the
nonperiodicity; see Remark~\ref{modelrem}. Since all norms on
$\C^{s-1}\times\R^t \simeq \R^{d-2}$ are equivalent, it suffices to prove the
assertion in case of the maximum norm on $\C^{s-1}\times\R^t$ with
respect to the absolute value on $\C$ and $\R$, respectively, rather than considering the Euclidean norm itself. But in that case, the assertion
follows immediately with $l:=1$ and $$c:=\operatorname{max}\big\{\lvert \sigma_{j}(\lambda)\rvert\,\big|
\, j\in\{2,\dots,s+t\}\big\}\,,$$ since the set
$\{\sigma_{2}(\lambda),\dots,\sigma_{s+t}(\lambda)\}$
of 
conjugates of $\lambda$ does not contain $\lambda$ itself. To see this, note
that $\sigma_j(\lambda)=\lambda$, where $j\in\{2,\dots,s+t\}$, implies
that $\sigma_j$ fixes $K^+$ whence $\sigma_j$  is the identity
or the complex conjugation, a contradiction; see~Definition~\ref{algmodel} and Remark~\ref{modelrem}. 
\end{proof}

\begin{lem}\label{dense}
Let $\varLambda$ be a $K$-algebraic model set
with star map $.^{\star}$ and let $d:=[K:\Q]$. Then
$\mathcal{O}_{K}^{\star}$ is dense in $\C^{s-1}\times\R^t\simeq\R^{d-2}$.  
\end{lem}
\begin{proof}
If $d=2$, one even has $\mathcal{O}_{K}^{\star}=\C^{s-1}\times\R^t=\{0\}$.
 Otherwise, choose a PV-number $\lambda$
of degree $d/2$ in $K^+$; cf.~Fact~\ref{pisot}. Since
$\mathcal{O}_{K}$ is a full $\Z$-module in $K$, the set
$\{\lambda^kz\,|\,z\in\mathcal{O}_{K}\}$ is a full
$\Z$-module in $K$ for any $k\in\N$ . Thus the set $$\{(\lambda^k
z,(m_{\lambda}^{\star})^k(z^{\star}))\,|\,z\in\mathcal{O}_{K}\}\,,$$
is a (full) lattice in
$\C^s\times\R^t\simeq\R^d$ for any $k\in\N$, where
$m_{\lambda}^{\star}$ is the $\Z$-module 
endomorphism of $\mathcal{O}_K^{\star}$ from Lemma~\ref{r2}; cf.~\cite[Ch.~2,
Sec.~3]{Bo}. In conjunction with Lemma~\ref{r2}, this implies that,
for any $\varepsilon>0$, the $\Z$-module $\mathcal{O}_{K}^{\star}$
contains an $\R$-basis of $\C^{s-1}\times\R^t$ whose elements have
norms $\leq\varepsilon$. The assertion follows.
\end{proof}

\begin{lem}\label{dilate}
Let $\varLambda$ be a $K$-algebraic model set.
 Then, for any finite set $F\subset K$, there is a homothety $h$ of
 the complex plane such that $h(F)\subset \varLambda$. Moreover, $h$
 can be chosen such that $h(z)=
\kappa z + v$, where $\kappa\in K^+$ is an algebraic integer with $\kappa\geq 1$ and
$v\in\varLambda$.
\end{lem}
\begin{proof}
Without loss of generality, we may assume that $\varLambda$ is of the
form $\varLambda(W)$ (see Definition~\ref{algmodel}) and that
$F\neq\varnothing$. Note that there is an $l\in
\mathbbm{N}$ such that $\{lz\,|\,z\in F\} \subset
\mathcal{O}_{K}$. Let $d:=[K:\Q]$ and let
$.^{\star}$ be the star map of $\varLambda$. If $d=2$, we are done by
 setting $h(z):=lz$. Otherwise, since $W$ has nonempty interior, Lemma~\ref{dense} shows the existence of a
suitable $z_{0}\in \mathcal{O}_{K}$ with $z_{0}^{\star}\in W^{\circ}$. Consider the open neighbourhood $V:= W^{\circ} -  z_{0}^{\star}$ of $0$ in
$\C^{s-1}\times\R^t$ and choose a PV-number $\lambda$
of degree $d/2$ in $K^+$;
cf.~Fact~\ref{pisot}. By virtue of Lemma~\ref{r2}, there is a
$k\in\mathbbm{N}$ such
that $$(m_{\lambda}^{\star})^{k}\big((lF)^{\star}\big)\subset V\,.$$
It
follows that $\{(\lambda^{k} z + z_{0})^{\star}\, |\, z\in lF\}\subset
W^{\circ}$ and, further, that $h(F)\subset \varLambda$,
where $h$ is the homothety given by $z
\mapsto (l\lambda^{k}) z + z_{0}$. The additional statement follows
immediately from the observation that $z_0\in\varLambda$. 
\end{proof}

\begin{prop}\label{cmsads}
$K$-algebraic model sets are algebraic Delone
sets. Moreover, any $K$-algebraic model set $\varLambda$
satisfies $K_{\varLambda}=K$.
\end{prop}
\begin{proof}
Since $K_{\varLambda}=K_{t+\varLambda}$ for any $t\in\C$, we may
assume that $\varLambda$ is of the
form $\varLambda(W)$ (see Definition~\ref{algmodel}). Any $K$-algebraic model set $\varLambda$ is a Delone 
set by Remark~\ref{modelrem}. Property (Alg) follows from the
observation that $K_{\varLambda}\subset K$ (recall that
$\varLambda-\varLambda\subset\mathcal{O}_K$ and that $\overline{K}=K$). Further,
property (Hom) is an immediate consequence of Lemma~\ref{dilate}.  Let
$\{\alpha_1,\dots,\alpha_d\}$ be a $\Q$-basis of $K/\Q$. By the additional statement of
Lemma~\ref{dilate} there is a nonzero element $\kappa\in K^+$ and a
point 
$v\in\varLambda$ such
that the
$\Q$-linear independent set
$\{\kappa\alpha_1,\dots,\kappa\alpha_d\}$ is contained in
$\varLambda-\{v\}\subset K_{\varLambda}$. Since
$K_{\varLambda}\subset K$, this shows that $K_{\varLambda}=K$.   
\end{proof}

\begin{rem}\label{okdirections}
As another immediate consequence of Lemma~\ref{dilate}, one verifies
that, for any $K$-algebraic model set $\varLambda$, the set of 
$\varLambda$-directions is precisely the set of
 $\mathcal{O}_{K}$-directions. 
\end{rem}

\begin{ex}\label{algex}
Standard examples of $n$-cyclotomic Delone sets are the $\Q(\zeta_n)$-algebraic
  model sets, where $n\geq 3$, which from now on are called {\em $n$-cyclotomic
  model sets}; cf.~Fact~\ref{gau} and Proposition~\ref{cmsads} (note
also that $\Q(\zeta_n)$ is obtained from $\Q(\zeta_n)^+$ by adjoining
the square root of the negative number $\zeta_n^2+\zeta_n^{-2}-2\in
\Q(\zeta_n)^+$, the latter being the discriminant of $X^2-(\zeta_n+\zeta_n^{-1})X+1$). These sets were also called {\em cyclotomic model sets with underlying
$\Z$-module $\Z[\zeta_n]$} in~\cite[Sec.~4.5]{H5}, since $\Z[\zeta_n]$
is the ring of integers in the $n$th cyclotomic field;
cf.~\cite[Thm.~ 2.6]{Wa}. The latter 
range from periodic examples like the fourfold square lattice ($n=4$)
or the sixfold triangular lattice ($n=3$) to nonperiodic
examples like the vertex set of the tenfold T\"ubingen triangle
tiling~\cite{bk1},  \cite{bk2} ($n=5$), the eightfold Ammann-Beenker tiling
of the plane~\cite{am},  \cite{bj},  \cite{ga} ($n=8$) or the twelvefold
shield tiling~\cite{ga}
($n=12$);
see~\cite[Fig.~1]{H4}, ~\cite[Fig.~2]{H5} and Fig.~\ref{fig:tilingupolygon} below for
illustrations. In general, for any divisor $m$ 
  of $\operatorname{lcm}(n,2)$, one can choose the window such that the
corresponding $n$-cyclotomic
  model sets have $m$-fold cyclic symmetry in the
  sense of symmetries of LI-classes, meaning that a discrete structure
  has a certain symmetry if the original and the transformed structure
  are locally indistinguishable; cf.~\cite{B} for details. Note that  the vertex sets of the famous Penrose tilings of the
plane fail to be $5$-cyclotomic model sets but can still be seen to be $5$-cyclotomic Delone sets; see~\cite{bh}
and references therein. 
\end{ex}

\section{A cyclotomic theorem}\label{cyc}

\begin{defi}\label{fmddefi}
Let $m\geq 4$ be a natural number. Set
$$
D_{m}:=\big\{(k_1,k_2,k_3,k_4)\in \mathbb{N}^4 \,\big|\, k_3<k_1\leq
k_2<k_4\leq m-1 \mbox{ and } k_1+k_2=k_3+k_4\big\}
$$ 
and define the function $f_{m}\,:\, D_{m}\rightarrow \C^*$ by
\begin{equation}\label{fmd}
f_{m}(k_1,k_2,k_3,k_4):=\frac{(1-\zeta_{m}^{k_1})(1-\zeta_{m}^{k_2})}{(1-\zeta_{m}^{k_3})(1-\zeta_{m}^{k_4})}.
\end{equation}
We further set $\mathcal{C}_m:=f_m(D_m)$ (note that $\mathcal{C}_m\subset
\mathcal{C}_{m'}$ for any multiple $m'$ of $m$) and
$\mathcal{C}:=\bigcup_{m\geq 4}\mathcal{C}_m$. Moreover, for a subset 
$K$ of $\C$, we set $\mathcal{C}(K):=\mathcal{C}\cap K$ and
$\mathcal{C}_m(K):=\mathcal{C}_m\cap K$.
\end{defi}

\begin{fact}\label{fmdg1}\cite[Lemma 3.1]{GG} 
Let $m\geq 4$. The function
$f_{m}$ is real-valued. Moreover, one has $f_{m}(d)>1$ for all $d\in
D_{m}$.\qed
\end{fact}

For our application to discrete tomography, we shall below show the
{\em finiteness} of the
set $\mathcal{C}(L)$ 
 for all real algebraic number fields $L$ and provide
 explicit results in the three cases $\Q(\zeta_{5})^+=\Q(\sqrt{5})$,
 $\Q(\zeta_8)^+=\Q(\sqrt{2})$ and 
 $\Q(\zeta_{12})^+=\Q(\sqrt{3})$. Gardner and Gritzmann showed the
 following result for the field $\Q=\Q(\zeta_3)^+=\Q(\zeta_4)^+$.

\begin{theorem}\cite[Lemma 3.8, Lemma 3.9 and Thm.~ 3.10]{GG}\label{intersectq}
$$
\mathcal{C}(\Q)=\mathcal{C}_{12}(\Q)=\big\{\tfrac{4}{3},\tfrac{3}{2},2,3,4\big\}\,.
$$
Moreover, all solutions of $f_{m}(d)=q\in\Q$, where $m\geq 4$ and
$d\in D_{m}$, are either given, up to multiplication of $m$
and $d$ by the same factor, by $m=12$ and one of the following
$$
\begin{array}{rlrl}
\textnormal{(i)}&d=(6,6,4,8),q=\frac{4}{3};&\textnormal{(ii)}&d=(6,6,2,10),q=4;\\
\textnormal{(iii)}&d=(4,8,3,9),q=\frac{3}{2};&\textnormal{(iv)}&d=(4,8,2,10),q=3;\\
\textnormal{(v)}&d=(4,4,2,6),q=\frac{3}{2};&\textnormal{(vi)}&d=(8,8,6,10),q=\frac{3}{2};\\
\textnormal{(vii)}&d=(4,4,1,7),q=3;&\textnormal{(viii)}&d=(8,8,5,11),q=3;\\
\textnormal{(ix)}&d=(3,9,2,10),q=2;&\textnormal{(x)}&d=(3,3,1,5),q=2;\\
\textnormal{(xi)}&d=(9,9,7,11),q=2;&&
\end{array}
$$
or by one of the following
$$
\begin{array}{rl}
\textnormal{(xii)}&d=(2k,s,k,k+s),q=2, \mbox{ where } s\geq 2, m=2s \mbox{
  and } 1\leq k\leq \frac{s}{2};\\
\textnormal{(xiii)}&d=(s,2k,k,k+s),q=2, \mbox{ where } s\geq 2, m=2s \mbox{
  and } \frac{s}{2}\leq k< s.\qed
\end{array}
$$
\end{theorem}

The next three lemmata are the key tools for our approach.

\begin{lem}\label{l1}
Let $a\in\R^*$. If $a=\tfrac{1+x}{1+y}$ for $x,y\in\mu\cup\{0\}$ with
$y\neq -1$ then
$a\in\{\tfrac{1}{2},1,2\}$.
\begin{proof}
It suffices to consider the cases $a=1+\omega$ and
$a=\tfrac{1+\omega_1}{1+\omega_2}$ with
$\omega,\omega_1,\omega_2\in\mu$ and $\omega_2\neq -1$. In the
first case, one has $\omega=a-1\in\mu(\R)=\{\pm 1\}$ whence $\omega=1$ (due
to $a\neq 0$) and $a=2$. In the second case, one has
$$
a=\bar{a}=\frac{1+\bar{\omega}_1}{1+\bar{\omega}_2}=\omega_2\omega_1^{-1}\frac{1+\omega_1}{1+\omega_2}=\omega_2\omega_1^{-1}a
$$ 
wherefore $\omega_1=\omega_2$ and $a=1$.
\end{proof}
\end{lem}

\begin{lem}[Comparison of coefficients]\label{l2}
Let $K\subset\C$ be a field, let $m\in\N$, and let $\zeta\in\mu$ with
$\zeta^m\in K$. Let
$a_0,\dots,a_{m-1},b_0,\dots,b_{m-1}\in K$ with
$$
\sum_{i=0}^{m-1}a_i\zeta^i=\sum_{i=0}^{m-1}b_i\zeta^i\,.
$$
Then one has $a_i=b_i$ for all $i=0,\dots,m-1$ if one of the following
conditions holds.
\begin{itemize}
\item[(a)]
$[K(\zeta) :K]=m$.
\item[(b)]
$[K(\zeta) :K]=m-1$ and at most $m-1$ of $a_0,\dots,a_{m-1},b_0,\dots,b_{m-1}$ are nonzero.
\end{itemize}
Moreover, if $[K(\zeta) :K]=m-1$ and $a_k-b_k\neq 0$ for some $k$ then
$|a_i-b_i|=|a_j-b_j|\neq 0$ for all $i,j$. 
\end{lem}
\begin{proof}
In case (a), the assertion follows immediately from the linear independence
of $1,\zeta,\dots,\zeta^{m-1}$ over $K$. If $[K(\zeta) :K]=m-1$, set $\omega:=\zeta^m\in K$. The minimum
polynomial $f\in K[X]$ of $\zeta$ over $K$ has degree $m-1$ and one
has $X^{m}-\omega=(X-\epsilon)f$ with $\epsilon\in K$, hence
$\omega=\epsilon^m$ (in particular, $\epsilon \in\mu(K)$) and
$$
f=\frac{X^m-\epsilon^m}{X-\epsilon}=\sum_{i=0}^{m-1}\epsilon^{m-1-i}X^i\,.
$$
If $\sum_{i=0}^{m-1}(a_i-b_i)\zeta^i=0$ then there is an element $c\in K$ with
$a_i=b_i+c\epsilon^{m-1-i}$ for all $i=0,\dots,m-1$. By assumption
(b) one has $a_i=0=b_i$ for some $i$. This implies $c=0$ and therefore
the assertion. For the additional statement, first observe that due to
$a_k\neq b_k$ for some $k$ one has $c\neq 0$. Thus 
$|a_i-b_i|=|c\epsilon^{m-1-i}|=|c|=|c\epsilon^{m-1-j}|=|a_j-b_j|\neq
0$ for all $i,j$.
\end{proof}

\begin{lem}\label{l3}
Let $K\subset\C$ be a field, let $m\in\N$, and let $\zeta\in\mu$ with
$\zeta^m\in K$. Further, let
$\omega_1,\omega_2,\omega_3,\omega_4\in\mu(K)$ and
$k_1,k_2,k_3,k_4\in\{0,\dots,m-1\}$ satisfy the following conditions.
\begin{itemize}
\item
$\operatorname{gcd}(k_i,m)=1$ for some $i\in\{1,2,3,4\}$.
\item
$k_1+k_2\equiv k_3+k_4 \pmod m$
\item
$\omega_3\zeta^{k_3},\omega_4\zeta^{k_4}\neq 1$ and $a:=\frac{(1-\omega_1\zeta^{k_1})(1-\omega_2\zeta^{k_2})}{(1-\omega_3\zeta^{k_3})(1-\omega_4\zeta^{k_4})}\in
K\cap (\R^*\setminus\{\pm 1 \})$.
\end{itemize}
Then one has $a\in\{\tfrac{1}{2},2\}$ if one of the following
conditions holds.
\begin{itemize}
\item[(a)]
$[K(\zeta) :K]=m$ and $m\geq 3$.
\item[(b)]
$[K(\zeta) :K]=m-1$ and $m\geq 5$.
\end{itemize}
\end{lem}
\begin{proof}
Without restriction, we may assume that
$\operatorname{gcd}(k_1,m)=1$. Then, for $i=2,3,4$,
there are $a_i,b_i\in\Z$ such that $k_i=a_ik_1+b_i m$ and, with 
$\zeta':=\zeta^{k_1}$, $\zeta^{k_i}=(\zeta')^{a_i}(\zeta^m)^{b_i}$. Since one has $(\zeta')^m\in K$, $K(\zeta')=K(\zeta)$ and 
$$
\frac{(1-\omega_1\zeta^{k_1})(1-\omega_2\zeta^{k_2})}{(1-\omega_3\zeta^{k_3})(1-\omega_4\zeta^{k_4})}=\frac{(1-\omega'_1\zeta')(1-\omega'_2\zeta'^{k'_2})}{(1-\omega'_3\zeta'^{k'_3})(1-\omega'_4\zeta'^{k'_4})}
$$
for suitable $\omega'_1,\omega'_2,\omega'_3,\omega'_4\in\mu(K)$ and
$k'_2,k'_3,k'_4\in\{0,\dots,m-1\}$ with $1+k'_2\equiv k'_3+k'_4 \pmod
m$, we may further assume that
$k_1=1$. We thus obtain
$$
1-\omega_1\zeta-\omega_2\zeta^{k_2}+\omega_1\omega_2\zeta^{k_2+1}=a-a\omega_3\zeta^{k_3}-a\omega_4\zeta^{k_4}+a\omega_3\omega_4\zeta^{k_3+k_4}\,,
$$
where, without restriction, $k_3\leq k_4$. From now on, let $[k]\in\{0,\dots,m-1\}$ denote the canonical
representative of the equivalence class of $k\in\Z$ modulo
$m$. We may finally write
$$
1-\omega_1\zeta-\omega_2\zeta^{k_2}+\omega_1\omega_2\omega\zeta^{[k_2+1]}=a-a\omega_3\zeta^{k_3}-a\omega_4\zeta^{[1+k_2-k_3]}+a\omega_3\omega_4\omega'\zeta^{[k_2+1]}
$$
with $k_3\leq [1+k_2-k_3]$ and suitable $\omega,\omega'\in\mu(K)$.

{\bf Case 1.} $k_2=0$. Then
$$
1-\omega_1\zeta-\omega_2+\omega_1\omega_2\omega\zeta=a-a\omega_3\zeta^{k_3}-a\omega_4\zeta^{[1-k_3]}+a\omega_3\omega_4\omega'\zeta
$$
If $k_3=0$ then $a=\frac{1-\omega_2}{1-\omega_3}$ by Lemma~\ref{l2}
and the
assertion follows from Lemma~\ref{l1}. The case $k_3=1$ cannot occur
(due to $k_3\leq [1-k_3]$), whereas $k_3\geq 2$ implies $a=1-\omega_2$ by Lemma~\ref{l2}. The assertion follows from Lemma~\ref{l1}.  

{\bf Case 2.} $k_2=1$. Then
$$
1-(\omega_1+\omega_2)\zeta+\omega_1\omega_2\omega\zeta^2=a-a\omega_3\zeta^{k_3}-a\omega_4\zeta^{[2-k_3]}+a\omega_3\omega_4\omega'\zeta^2
$$
If $k_3=0$ then  $a=\frac{1}{1-\omega_3}$ by
Lemma~\ref{l2} and the assertion follows from
 Lemma~\ref{l1}. If $k_3=1$ then Lemma~\ref{l2} implies $a=1$, which
 is excluded by assumption. The
 case $k_3=2$ is impossible (due to $k_3\leq [2-k_3]$). Let $k_3\geq 3$
 (hence $m\geq 4$). Under condition~(a), this implies $a=1$ by
 Lemma~\ref{l2}, which
 is excluded by assumption. Under
 condition~(b), $k_3= 3$ implies
$$
1-(\omega_1+\omega_2)\zeta+\omega_1\omega_2\omega\zeta^{2}=a-a\omega_3\zeta^{3}-a\omega_4\zeta^{m-1}+a\omega_3\omega_4\omega'\zeta^{2}
$$ 
with $m-1\geq 4$ (due to $m\geq 5$). The
additional statement of Lemma~\ref{l2} implies $m=5$ and $|1-a|=|a\omega_4|=|a|$, wherefore $a=1/2$. If $k_3\geq 4$
then $m\geq 6$ (due to $k_3\leq [2-k_3]$) and Lemma~\ref{l2} implies $a=1$, which
 is excluded by assumption.  

{\bf Case 3.} $k_2\in\{2,\dots,m-2\}$ (hence $m\geq 4$ and $2\leq
k_2<k_2+1\leq m-1$). Then
$$
(1-a)-\omega_1\zeta-\omega_2\zeta^{k_2}+(\omega_1\omega_2\omega-a\omega_3\omega_4\omega')\zeta^{k_2+1}=-a\omega_3\zeta^{k_3}-a\omega_4\zeta^{[1+k_2-k_3]}
$$
Under condition~(a), Lemma~\ref{l2} shows that this is impossible, since there are at least three 
nontrivial coefficients on the left-hand side and at
most two nontrivial coefficients on the right-hand side of this
equation. Under condition~(b) (hence $m\geq 5$), $k_3=0$
implies $a-1=a\omega_3$ by Lemma~\ref{l2} wherefore
$a=\frac{1}{1-\omega_3}$ and the assertion follows from
Lemma~\ref{l1}. If $k_3=1$ then $a=1$ by Lemma~\ref{l2}, which
 is excluded by assumption.  If $k_3\geq 2$
and $m\geq 7$ then $a=1$ by Lemma~\ref{l2}, which
 is excluded by assumption. Employing the additional statement of
 Lemma~\ref{l2}, we shall now see that the
missing cases ($k_3\geq 2$ and $m\in\{5,6\}$) are either impossible or yield 
$|1-a|=1$ and thus $a=2$ (due to $a\neq 0$). In fact, $m=5$
and $k_3= 2$ imply $k_2=3$ (due to $k_3\leq [1+k_2-k_3]$) and, further,
$|1-a|=|\omega_1|=1$.  The case $m=5$
and $k_3= 3$ cannot occur (due to $k_3\leq [1+k_2-k_3]$). If $m=5$
and $k_3= 4$ then $k_2=2$ (due to $k_3\leq [1+k_2-k_3]$) and, further,
$|1-a|=|\omega_1|=1$. If $m=6$
and $k_3= 2$ then $k_2\in\{3,4\}$ (due to $k_3\leq [1+k_2-k_3]$). The
case $k_2= 3$ is impossible, whereas the case $k_2= 4$ yields $|1-a|=|\omega_1|=1$. The case $m=6$
and $k_3= 3$ is impossible (due to $k_3\leq [1+k_2-k_3]$). The case $m=6$
and $k_3= 4$ implies $k_2=2$ (due to $k_3\leq [1+k_2-k_3]$) and, further,  $|1-a|=|\omega_1|=1$. Finally, the case $m=6$
and $k_3= 5$ implies $k_2= 3$ (due to $k_3\leq [1+k_2-k_3]$) and, once
again,  $|1-a|=|\omega_1|=1$.

{\bf Case 4.} $k_2=m-1$. Then
$$
(1+\omega_1\omega_2\omega)-\omega_1\zeta-\omega_2\zeta^{m-1}=a(1+\omega_3\omega_4\omega')-a\omega_3\zeta^{k_3}-a\omega_4\zeta^{[m-k_3]}
$$
Under condition~(a), Lemma~\ref{l2} implies
$\{k_3,[m-k_3]\}=\{1,m-1\}$, wherefore $k_3=1$ and $[m-k_3]=m-1$ (due
to $k_3\leq [m-k_3]$). Further, Lemma~\ref{l2} yields
$a=\omega_2/\omega_4$, a contradiction (due to $|a|\neq 1$). By the
additional statement of Lemma~\ref{l2}, condition~(b) (hence $m\geq 5$)
implies $m=5$, $k_3=2$ and, further, $|a\omega_3|=|\omega_2|=1$, a contradiction
(due to $|a|\neq 1$).
\end{proof}

We are now in a position to prove the following extension of Theorem~\ref{intersectq}.   

\begin{theorem}\label{t1}
For $n\in\N$, one has
$$\mathcal{C}(\Q(\zeta_n)^+)=\mathcal{C}_{\operatorname{lcm}(2n,12)}(\Q(\zeta_n)^+)\,.$$
In particular, the last set is finite. Moreover, all solutions of $f_{m}(d)\in\Q(\zeta_n)^+$, where $m\geq 4$ and
$d\in D_{m}$, are either of the form
(xii) or (xiii) of Theorem~\ref{intersectq} or are  given, up to multiplication of $m$
and $d$ by the same factor, by $m=\operatorname{lcm}(2n,12)$ and $d$
from a finite list. 
\end{theorem}
\begin{proof}
Since $\Q(\zeta_n)^+=\Q(\zeta_{2n})^+$ for odd $n$ it suffices to
consider the case where $n$ is even (hence
$\operatorname{lcm}(2n,12)=\operatorname{lcm}(2n,3)$). Let $m\geq 4$
and $d:=(k_1,k_2,k_3,k_4)\in D_{m}$ such that
$$a:=f_m(d)=\frac{(1-\zeta_{m}^{k_1})(1-\zeta_{m}^{k_2})}{(1-\zeta_{m}^{k_3})(1-\zeta_{m}^{k_4})}\in\Q(\zeta_n)^+$$
Recall that $a>1$ by Fact~\ref{fmdg1}.  We may
assume that $\operatorname{gcd}(m,k_1,k_2,k_3,k_4)=1$.  By virtue of Theorem~\ref{intersectq}, we
may also assume that $a\not\in\Q$. Observe that
$$a\in\Q(\zeta_n)^+\cap \Q(\zeta_m)^+= \Q(\zeta_{\operatorname{gcd}(m,n)})^+$$
Claims 1 and 2 below show that $\operatorname{lcm}(2n,3)$ is a
multiple of $m$, hence the assertion.

 {\bf Claim 1.} Let $p$ be an odd prime number and assume that
 $\operatorname{ord}_p(n)<\operatorname{ord}_p(m)$. Then one has $p=3$,
$\operatorname{ord}_p(m)=1$ and $\operatorname{ord}_p(n)=0$.

To see this, set $r:=\operatorname{ord}_p(m)\geq 1$, 
$K:=\Q(\zeta_{m/p})$ and note that $
\Q(\zeta_{\operatorname{gcd}(m,n)})\subset K$. Let $m=p^rm'$,
where $\operatorname{gcd}(p,m')=1$. Then, for $i=1,2,3,4$,
there are $a_i,b_i\in\Z$ such that $k_i=a_ip+b_i m'$ and, further, 
$\zeta_m^{k_i}=\zeta_{m/p}^{a_i}\zeta_{p^r}^{b_i}$. Since $\zeta_{p^r}^p=\zeta_{p^{r-1}}\in K$ one has 
$$
a=\frac{(1-\omega_1\zeta_{p^r}^{l_1})(1-\omega_2\zeta_{p^r}^{l_2})}{(1-\omega_3\zeta_{p^r}^{l_3})(1-\omega_4\zeta_{p^r}^{l_4})}
$$ 
for suitable $\omega_1,\omega_2,\omega_3,\omega_4\in\mu(K)$ and
$l_1,l_2,l_3,l_4\in\{0,\dots,p-1\}$ with $\operatorname{gcd}(l_i,p)=1$
for some $i\in\{1,2,3,4\}$ and $l_1+l_2\equiv l_3+l_4 \pmod
p$. Further, by Fact~\ref{gau}, one has
$$
[K(\zeta_{p^r}):K]=[\Q(\zeta_{m}):\Q(\zeta_{m/p})]=\frac{\phi(p^r)}{\phi(p^{r-1})}=\left\{\begin{array}{ll}
p-1 & \mbox{if $r=1$;}\\
p & \mbox{if $r\geq 2$.}
\end{array}\right.
$$
Lemma~\ref{l3} implies both for $p\geq
5$ and $r\geq 2$ that $a=2$, a contradiction. Therefore $p=3$,
$r=\operatorname{ord}_p(m)=1$ and consequently $\operatorname{ord}_p(n)=0$.

{\bf Claim 2.} $\operatorname{ord}_2(m)\leq
\operatorname{ord}_2(n)+1$. 

Assume that $r:=\operatorname{ord}_2(m)\geq
\operatorname{ord}_2(n)+2\geq 3$. Set
$K:=\Q(\zeta_{m/4})$ and note that $\Q(\zeta_{\operatorname{gcd}(m,n)})\subset
K$. As above, since $\zeta_{2^{r}}^4=\zeta_{2^{r-2}}\in K$, one has 
$$
a=\frac{(1-\omega_1\zeta_{2^r}^{l_1})(1-\omega_2\zeta_{2^r}^{l_2})}{(1-\omega_3\zeta_{2^r}^{l_3})(1-\omega_4\zeta_{2^r}^{l_4})}
$$ 
for suitable $\omega_1,\omega_2,\omega_3,\omega_4\in\mu(K)$ and
$l_1,l_2,l_3,l_4\in\{0,1,2,3\}$ with $\operatorname{gcd}(l_i,4)=1$
for some $i\in\{1,2,3,4\}$ and $l_1+l_2\equiv l_3+l_4 \pmod
4$. Further, by Fact~\ref{gau}, one has
 $$[K(\zeta_{2^{s}}):K]=[\Q(\zeta_{m}):\Q(\zeta_{m/4})]=\frac{\phi(2^r)}{\phi(2^{r-2})}=4\,.$$ 
Lemma~\ref{l3} now implies $a=2$, a contradiction. This proves the
claim. 
\end{proof}

\begin{rem}
Similar to the proof of Theorem~\ref{t1}, one can also use
Lemma~\ref{l3} to give another proof of the fact shown in~\cite{GG} that all solutions of $f_{m}(d)\in\Q\setminus\{2\}$, where $m\geq 4$ and
$d\in D_{m}$, are given, up to multiplication of $m$
and $d$ by the same factor, by $m=12$. Thus the number $2$ plays a special role
in this context. Indeed this number leads to infinite families of
solutions (see Theorem~\ref{intersectq}(xii)-(xiii) above) that can be
found by using the $2$-adic valuation;
cf.~\cite{GG} for details.   
\end{rem}

One even has the following result, which improves~\cite[Thm.~ 4.19]{H5}.

\begin{theorem}\label{algcoro}
For any real algebraic number field $L$, the set
$\mathcal{C}(L)$ is finite. Moreover, there is a number
$m_L\in\N$ such that all solutions of $f_{m}(d)\in L$, where $m\geq 4$ and
$d\in D_{m}$, are either of the form
(xii) or (xiii) of Theorem~\ref{intersectq} or are  given, up to multiplication of $m$
and $d$ by the same factor, by $m=m_L$ and $d$
from a finite list. 
\end{theorem}

\begin{proof} 
The finiteness of $L/\Q$ together with  the identity
$$\Q(\mu)^+=\bigcup_{n\in\N}\Q(\zeta_n)^+$$ implies that $L \cap \Q(\mu)^+ = L\cap \Q(\zeta_n)^+$ for
some $n\in\N$. Since $\mathcal{C} \subset \Q(\mu)^+$ by Fact~\ref{fmdg1} it follows that 
$$
\mathcal{C}(L)  = L \cap \mathcal{C}  = L \cap \mathcal{C} \cap \Q(\mu)^+=  L\cap \mathcal{C}\cap \Q(\zeta_n)^+ \subset
\mathcal{C}(\Q(\zeta_n)^+)\,.
$$
By virtue of Theorem~\ref{t1}, the assertion follows with $m_L:=\operatorname{lcm}(2n,12)$.
\end{proof} 

\begin{coro}\label{coro8125}
\begin{itemize}
\item[(a)]
\begin{eqnarray*}
\mathcal{C}(\Q(\sqrt{5}))&=&\mathcal{C}_{60}(\Q(\sqrt{5}))\\&=&
\big\{\tfrac{10-2\sqrt{5}}{5},\tfrac{\sqrt{5}}{2},\tfrac{9-3\sqrt{5}}{2},\tfrac{5+3\sqrt{5}}{10},\tfrac{5+\sqrt{5}}{6},-1+\sqrt{5},\tfrac{3+\sqrt{5}}{4},\tfrac{4}{3},
\tfrac{5-\sqrt{5}}{2},\tfrac{5+\sqrt{5}}{5},\\
&&\hphantom{\big\{}\tfrac{3}{2},6-2\sqrt{5},\tfrac{1+\sqrt{5}}{2},\tfrac{5+\sqrt{5}}{4},\tfrac{-3+3\sqrt{5}}{2},\tfrac{5+2\sqrt{5}}{5},2,\tfrac{2+\sqrt{5}}{2},\tfrac{15+3\sqrt{5}}{10},\sqrt{5},\\
&&\hphantom{\big\{}\tfrac{3+\sqrt{5}}{2},\tfrac{10+2\sqrt{5}}{5},3,1+\sqrt{5},\tfrac{5+\sqrt{5}}{2},4,2+\sqrt{5},3+\sqrt{5},\tfrac{5+3\sqrt{5}}{2},\tfrac{7+3\sqrt{5}}{2},\\
&&\hphantom{\big\{} \tfrac{9+3\sqrt{5}}{2},5+2\sqrt{5},6+2\sqrt{5}\big\}
\end{eqnarray*}
Moreover, all solutions of $f_{m}(d)\in\Q(\sqrt{5})$, where $m\geq 4$ and
$d\in D_{m}$, are either of the form
(xii) or (xiii) of Theorem~\ref{intersectq} or are  given,  up to multiplication of $m$
and $d$ by the same factor, by $m=60$ and $d$ from the following list.
$$
\begin{tabular}{|r|c|r|c|r|c|r|c|}
\hline
$1$&$(12, 36, 6, 42)$&$2$&$(24, 24, 9, 39)$&$3$& $(24, 48, 18, 54)$&$4$& $(36, 36, 21, 51)$\\\hline $5$& $(24, 24, 18, 
30)$&$6$& $(36, 36, 30, 42)$&$7$&  $(4, 8, 2, 10)$&$8$& $(5, 25, 3, 27)$\\\hline $9$& $(6, 42, 4, 44)$&$10$& 
$(8, 14, 4, 18)$&$11$& $(8, 32, 5, 35)$&$12$&  $(8, 50, 6, 52)$\\\hline $13$& $(9, 21, 5, 25)$&$14$& $(9, 
39, 6, 42)$&$15$& $(10, 10, 4, 16)$&$16$& $(10, 28, 6, 32)$\\\hline $17$&  $(10, 52, 8, 54)$&$18$& $(12, 
18, 6, 24)$&$19$& $(14, 26, 8, 32)$&$20$& $(14, 34, 9, 39)$\\\hline $21$& $(14, 42, 10, 46)$&$22$&  $(16, 
32, 10, 38)$&$23$& $(18, 18, 8, 28)$&$24$& $(18, 26, 10, 34)$\\\hline $25$& $(18, 36, 12, 42)$&$26$& 
$(18, 46, 14, 50)$&$27$&  $(18, 54, 16, 56)$&$28$& $(21, 51, 18, 54)$\\\hline $29$& $(24, 24, 12, 
36)$&$30$& $(24, 42, 18, 48)$&$31$& $(26, 32, 16, 42)$&$32$&  $(26, 46, 21, 51)$\\\hline $33$& $(28, 34, 
18, 44)$&$34$& $(28, 44, 22, 50)$&$35$& $(28, 52, 25, 55)$&$36$& $(32, 50, 28, 54)$\\\hline $37$&  $(34, 
42, 26, 50)$&$38$& $(34, 46, 28, 52)$&$39$& $(35, 55, 33, 57)$&$40$& $(36, 36, 24, 48)$\\\hline $41$& 
$(39, 51, 35, 55)$&$42$&  $(42, 42, 32, 52)$&$43$& $(42, 48, 36, 54)$&$44$& $(46, 52, 42, 
56)$\\\hline $45$& $(50, 50, 44, 56)$&$46$& $(52, 56, 50, 58)$&$47$&  $(18, 18, 6, 30)$&$48$& $(42, 42, 
30, 54)$\\\hline $49$& $(4, 52, 2, 54)$&$50$& $(5, 35, 2, 38)$&$51$& $(6, 18, 2, 22)$&$52$&  $(8, 10, 2, 
16)$\\\hline $53$& $(8, 28, 3, 33)$&$54$& $(8, 46, 4, 50)$&$55$& $(8, 56, 6, 58)$&$56$& $(9, 21, 3, 27)$\\\hline $57$&  
$(9, 39, 4, 44)$&$58$& $(10, 32, 4, 38)$&$59$& $(10, 50, 6, 54)$&$60$& $(12, 42, 6, 48)$\\\hline $61$& 
$(14, 18, 4, 28)$&$62$&  $(14, 26, 5, 35)$&$63$& $(14, 34, 6, 42)$&$64$& $(14, 52, 10, 56)$\\\hline $65$& 
$(16, 28, 6, 38)$&$66$& $(18, 24, 6, 36)$&$67$&  $(18, 34, 8, 44)$&$68$& $(18, 42, 10, 50)$\\\hline $69$& 
$(18, 48, 12, 54)$&$70$& $(21, 51, 16, 56)$&$71$& $(24, 36, 12, 48)$&$72$&  $(25, 55, 22, 
58)$\\\hline $73$& $(26, 28, 10, 44)$&$74$& $(26, 42, 16, 52)$&$75$& $(26, 46, 18, 54)$&$76$& $(28, 50, 
22, 56)$\\\hline $77$&  $(32, 34, 16, 50)$&$78$& $(32, 44, 22, 54)$&$79$& $(32, 52, 27, 57)$&$80$& $(34, 
46, 25, 55)$\\\hline $81$& $(36, 42, 24, 54)$&$82$&  $(39, 51, 33, 57)$&$83$& $(42, 46, 32, 56)$&$84$& 
$(42, 54, 38, 58)$\\\hline $85$& $(50, 52, 44, 58)$&$86$& $(12, 12, 2, 22)$&$87$&  $(12, 24, 3, 33)$&$88$& 
$(12, 36, 4, 44)$\\\hline $89$& $(12, 48, 6, 54)$&$90$& $(24, 24, 6, 42)$&$91$& $(24, 36, 10, 50)$&$92$&  
$(24, 48, 16, 56)$\\\hline $93$& $(36, 36, 18, 54)$&$94$& $(36, 48, 27, 57)$&$95$& $(48, 48, 38, 
58)$&$96$& $(8, 28, 6, 30)$\\\hline $97$&  $(14, 26, 10, 30)$&$98$& $(18, 24, 12, 30)$&$99$& $(18, 42, 15, 
45)$&$100$& $(32, 52, 30, 54)$\\\hline $101$& $(34, 46, 30, 50)$&$102$&  $(36, 42, 30, 48)$&$103$& $(12, 24, 
6, 30)$&$104$& $(24, 36, 15, 45)$\\\hline $105$& $(36, 48, 30, 54)$&$106$& $(15, 45, 12, 48)$&$107$&  $(18, 
30, 12, 36)$&$108$& $(30, 42, 24, 48)$\\\hline $109$& $(24, 36, 20, 40)$&$110$& $(10, 30, 8, 32)$&$111$& 
$(15, 15, 9, 21)$&$112$&  $(18, 30, 14, 34)$\\\hline $113$& $(24, 30, 18, 36)$&$114$& $(30, 36, 24, 
42)$&$115$& $(30, 42, 26, 46)$&$116$& $(30, 50, 28, 52)$\\\hline $117$&  $(45, 45, 39, 51)$&$118$& $(15, 15, 
3, 27)$&$119$& $(18, 30, 6, 42)$&$120$& $(30, 42, 18, 54)$\\\hline $121$& $(45, 45, 33, 57)$&$122$&  $(8, 32, 
2, 38)$&$123$& $(14, 34, 4, 44)$&$124$& $(18, 18, 3, 33)$\\\hline $125$& $(18, 36, 6, 48)$&$126$& $(24, 42, 
12, 54)$&$127$&  $(26, 46, 16, 56)$&$128$& $(28, 52, 22, 58)$\\\hline $129$& $(42, 42, 27, 57)$&$130$& $(10, 
30, 2, 38)$&$131$& $(15, 45, 6, 54)$&$132$&  $(18, 30, 4, 44)$\\\hline $133$& $(24, 30, 6, 48)$&$134$& $(30, 
36, 12, 54)$&$135$& $(30, 42, 16, 56)$&$136$& $(30, 50, 22, 58)$\\\hline $137$&  $(24, 36, 6, 54)$&$138$& 
$(30, 30, 6, 54)$&$139$& $(30, 30, 18, 42)$&$140$& $(12, 12, 6, 18)$\\\hline $141$& $(12, 24, 8, 28)$&$142$&  
$(12, 36, 9, 39)$&$143$& $(12, 48, 10, 50)$&$144$& $(24, 24, 14, 34)$\\\hline $145$& $(24, 36, 18, 
42)$&$146$& $(24, 48, 21, 51)$&$147$&  $(36, 36, 26, 46)$&$148$& $(36, 48, 32, 52)$\\\hline $149$& $(48, 48, 
42, 54)$&$150$& $(14, 26, 2, 38)$&$151$& $(18, 42, 6, 54)$&$152$&  $(34, 46, 22, 58)$\\\hline $153$& $(14, 34, 
12, 36)$&$154$& $(18, 18, 12, 24)$&$155$& $(26, 46, 24, 48)$&$156$& $(42, 42, 36, 48)$\\\hline $157$&  $(18, 
42, 12, 48)$&$158$& $(20, 20, 2, 38)$&$159$& $(20, 40, 6, 54)$&$160$& $(40, 40, 22, 58)$\\\hline $161$& $(20, 
20, 8, 32)$&$162$&  $(40, 40, 28, 52)$&$163$& $(20, 20, 14, 26)$&$164$& $(20, 40, 18, 42)$\\\hline $165$& 
$(40, 40, 34, 46)$&$166$& $(20, 40, 12, 48)$&$167$&  $(24, 24, 4, 44)$&$168$& $(36, 36, 16, 
56)$\\\hline $169$& $(30, 30, 12, 48)$&$170$& $(30, 30, 24, 36)$&$171$&$(30,30,20,40)$&$172$&$(30,30,10,50)$\\\hline $173$&$(20,40,15,45)$&$174$&$(20,40,10,50)$&$175$&$(20,20,10,30)$&$176$&$(40,40,30,50)$\\\hline $177$&$(20,20,5,35)$&$178$&$(40,40,25,55)$&$179$&$(15,45,10,50)$&$180$&$(15,15,5,25)$\\\hline $181$&$(45,45,35,55)$&&&&&&\\
\hline
\end{tabular}
$$
\item[(b)] 
\begin{eqnarray*}
\mathcal{C}(\Q(\sqrt{2}))&=&\mathcal{C}_{48}(\Q(\sqrt{2}))\\&=&
\big\{\tfrac{2+\sqrt{2}}{3},4-2\sqrt{2},\tfrac{1+\sqrt{2}}{2},-3+3\sqrt{2},\tfrac{4}{3},\sqrt{2},\tfrac{3}{2},\tfrac{2+\sqrt{2}}{2},2,,1+\sqrt{2},
  \\
&&\hphantom{\big\{}3,2+\sqrt{2}, 4,\tfrac{6+3\sqrt{2}}{2},3+2\sqrt{2},4+2\sqrt{2},4+3\sqrt{2}\big\}\,.
\end{eqnarray*}
Moreover, all solutions of $f_{m}(d)\in\Q(\sqrt{2})$, where $m\geq 4$ and
$d\in D_{m}$, are either of the form
(xii) or (xiii) of Theorem~\ref{intersectq} or are  given,  up to multiplication of $m$
and $d$ by the same factor, by $m=48$ and $d$ from the following list.
$$
\begin{tabular}{|r|c|r|c|r|c|r|c|}
\hline
$1$&$(6, 18, 4, 20)$&$2$& $(10, 36, 8, 38)$&$3$& $(12, 12, 6, 18)$&$4$& $(12, 22, 8, 
  26)$\\\hline $5$& $(12, 30, 9, 33)$&$6$& $(12, 38, 10, 40)$&$7$& $(18, 18, 10, 26)$&$8$& $(18, 24, 
  12, 30)$\\\hline $9$& $(18, 36, 15, 39)$&$10$& $(24, 30, 18, 36)$&$11$& $(26, 36, 22, 40)$&$12$& $(30, 
  30, 22, 38)$\\\hline $13$& $(30, 42, 28, 44)$&$14$& $(36, 36, 30, 42)$&$15$& $(4, 10, 2, 12)$&$16$& $(8,
   40, 6, 42)$\\\hline $17$& $(9, 33, 6, 36)$&$18$& $(10, 26, 6, 30)$&$19$& $(12, 18, 6, 24)$&$20$& $(15, 
  39, 12, 42)$\\\hline $21$& $(18, 30, 12, 36)$&$22$& $(20, 26, 12, 34)$&$23$& $(22, 28, 14, 
  36)$&$24$& $(22, 38, 18, 42)$\\\hline $25$& $(30, 36, 24, 42)$&$26$& $(38, 44, 36, 46)$&$27$& $(10, 22, 
  8, 24)$&$28$& $(18, 18, 12, 24)$\\\hline $29$& $(26, 38, 24, 40)$&$30$& $(30, 30, 24, 36)$&$31$& $(18, 
  30, 16, 32)$&$32$& $(4, 38, 2, 40)$\\\hline $33$& $(8, 8, 2, 14)$&$34$& $(9, 15, 3, 21)$&$35$& $(10, 22,
   4, 28)$&$36$& $(10, 44, 8, 46)$\\\hline $37$& $(12, 30, 6, 36)$&$38$& $(18, 18, 6, 30)$&$39$& $(18, 36,
   12, 42)$&$40$& $(20, 22, 8, 34)$\\\hline $41$& $(26, 28, 14, 40)$&$42$& $(26, 38, 20, 44)$&$43$& $(30, 
  30, 18, 42)$&$44$& $(33, 39, 27, 45)$\\\hline $45$& $(40, 40, 34, 46)$&$46$& $(6, 30, 2, 
  34)$&$47$& $(10, 12, 2, 20)$&$48$& $(12, 18, 3, 27)$\\\hline $49$& $(12, 26, 4, 34)$&$50$& $(12, 36, 6, 
  42)$&$51$& $(18, 24, 6, 36)$&$52$& $(18, 30, 8, 40)$\\\hline $53$& $(18, 42, 14, 46)$&$54$& $(22, 36, 
  14, 44)$&$55$& $(24, 30, 12, 42)$&$56$& $(30, 36, 21, 45)$\\\hline $57$& $(36, 38, 28, 46)$&$58$& $(10, 
  26, 2, 34)$&$59$& $(18, 30, 6, 42)$&$60$& $(22, 38, 14, 46)$\\\hline $61$& $(12, 24, 2, 34)$&$62$& $(24,
   24, 6, 42)$&$63$& $(24, 36, 14, 46)$&$64$& $(12, 24, 10, 26)$\\\hline $65$& $(24, 24, 18, 
  30)$&$66$& $(24, 36, 22, 38)$&$67$& $(16, 16, 10, 22)$&$68$& $(32, 32, 26, 38)$\\\hline $69$& $(18, 18, 
  2, 34)$&$70$& $(30, 30, 14, 46)$&$71$& $(16, 32, 6, 42)$&$72$&$(24,24,16,32)$\\\hline $73$&$(24,24,8,40)$&$74$&$(16,32,12,36)$&$75$&$(16,32,8,40)$&$76$&$(16,16,8,24)$\\\hline $77$&$(32,32,24,40)$&$78$&$(16,16,4,28)$&$79$&$(32,32,20,44)$&$80$&$(12,36,8,40)$\\\hline $81$&$(12,12,4,20)$&$82$&$(36,36,28,44)$&&&&\\
\hline
\end{tabular}
$$
\item[(c)]
\begin{eqnarray*}
\mathcal{C}(\Q(\sqrt{3}))&=&\mathcal{C}_{24}(\Q(\sqrt{3}))\\&=&
\big\{8-4\sqrt{3},\tfrac{3+2\sqrt{3}}{6},\tfrac{-3+3\sqrt{3}}{2},\tfrac{3+\sqrt{3}}{4},\tfrac{2+\sqrt{3}}{3},3-\sqrt{3},\tfrac{4}{3},\tfrac{1+\sqrt{3}}{2},-2+2\sqrt{3},\\
&&\hphantom{\big\{}
\tfrac{3}{2},\tfrac{3+\sqrt{3}}{3},\sqrt{3},\tfrac{2+\sqrt{3}}{2},2,\tfrac{3+2\sqrt{3}}{3},\tfrac{3+\sqrt{3}}{2},1+\sqrt{3},3,\tfrac{6+2\sqrt{3}}{3},2+\sqrt{3},4,\\
&&\hphantom{\big\{}3+\sqrt{3},\tfrac{5+3\sqrt{3}}{2},3+2\sqrt{3},4+2\sqrt{3},6+3\sqrt{3}, 7+4\sqrt{3},8+4\sqrt{3}\big\}\,.
\end{eqnarray*}
Moreover, all solutions of $f_{m}(d)\in\Q(\sqrt{3})$, where $m\geq 4$ and
$d\in D_{m}$, are either of the form
(xii) or (xiii) of Theorem~\ref{intersectq} or are  given,  up to multiplication of $m$
and $d$ by the same factor, by $m=24$ and $d$ from the following list.
$$
\begin{tabular}{|r|c|r|c|r|c|r|c|}
\hline
$1$&$(4, 8, 2, 10)$&$2$&$(6, 16, 4, 18)$&$3$&$(8, 10, 4, 14)$&$4$&$(8, 14, 5, 17)$\\\hline $5$&$(8, 
18, 6, 20)$&$6$&$(10, 16, 7, 19)$&$7$&$(14, 16, 10, 20)$&$8$&$(16, 20, 14, 22)$\\\hline $9$&$(3, 
15, 2, 16)$&$10$&$(4, 6, 2, 8)$&$11$&$(6, 14, 4, 16)$&$12$&$(9, 21, 8, 22)$\\\hline $13$&$(10, 12, 6, 
16)$&$14$&$(10, 18, 8, 20)$&$15$&$(12, 14, 8, 18)$&$16$&$(18, 20, 16, 22)$\\\hline $17$&$(10, 14, 8, 
16)$&$18$&$(4, 4, 2, 6)$&$19$&$(4, 14, 3, 15)$&$20$&$(5, 7, 3, 9)$\\\hline $21$&$(5, 17, 4, 18)$&$22$&$(6, 
10, 4, 12)$&$23$&$(7, 19, 6, 20)$&$24$&$(8, 14, 6, 16)$\\\hline $25$&$(10, 10, 6, 14)$&$26$&$(10, 16, 
8, 18)$&$27$&$(10, 20, 9, 21)$&$28$&$(14, 14, 10, 18)$\\\hline $29$&$(14, 18, 12, 20)$&$30$&$(17, 19, 
15, 21)$&$31$&$(20, 20, 18, 22)$&$32$&$(4, 10, 2, 12)$\\\hline $33$&$(10, 14, 6, 18)$&$34$&$(14, 20, 
12, 22)$&$35$&$(4, 16, 2, 18)$&$36$&$(6, 8, 2, 12)$\\\hline $37$&$(8, 10, 3, 15)$&$38$&$(8, 14, 4, 
18)$&$39$&$(8, 20, 6, 22)$&$40$&$(10, 16, 6, 20)$\\\hline $41$&$(14, 16, 9, 21)$&$42$&$(16, 18, 12, 
22)$&$43$&$(8, 10, 6, 12)$&$44$&$(14, 16, 12, 18)$\\\hline $45$&$(3, 9, 1, 11)$&$46$&$(4, 18, 2, 20)$&$47$&
$(6, 10, 2, 14)$&$48$&$(6, 20, 4, 22)$\\\hline $49$&$(10, 12, 4, 18)$&$50$&$(12, 14, 6, 20)$ &$51$&
$(14, 18, 10, 22)$&$52$&$(15, 21, 13, 23)$\\\hline $53$&$(4, 10, 1, 13)$&$54$&$(4, 20, 2, 22)$&$55$&
$(5, 7, 1, 11)$&$56$&$(5, 17, 2, 20)$\\\hline $57$&$(6, 14, 2, 18)$&$58$&$(7, 19, 4, 22)$&$59$&$(8, 
10, 2, 16)$&$60$&$(10, 14, 4, 20)$\\\hline $61$&$(10, 18, 6, 22)$&$62$&$(14, 16, 8, 22)$&$63$&$(14, 
20, 11, 23)$&$64$&$(17, 19, 13, 23)$\\\hline $65$&$(6, 8, 1, 13)$&$66$&$(6, 16, 2, 20)$&$67$&$(8, 12, 
2, 18)$&$68$&$(8, 18, 4, 22)$\\\hline $69$&$(12, 16, 6, 22)$&$70$&$(16, 18, 11, 23)$&$71$&$(6, 8, 4, 
10)$&$72$&$(6, 16, 5, 17)$\\\hline $73$&$(8, 12, 6, 14)$&$74$&$(8, 18, 7, 19)$&$75$&$(12, 16, 10, 
18)$&$76$&$(16, 18, 14, 20)$\\\hline $77$&$(6, 12, 4, 14)$&$78$&$(12, 18, 10, 20)$&$79$&$(6, 18, 2, 
22)$&$80$&$(10, 12, 2, 20)$\\\hline $81$&$(12, 14, 4, 22)$&$82$&$(8, 14, 2, 20)$&$83$&$(10, 16, 4, 
22)$&$84$&$(8, 16, 2, 22)$\\\hline $85$&$(10, 14, 2, 22)$&$86$&$(12, 12, 2, 22)$&$87$&$(12, 12, 10, 
14)$&$88$&$(4, 14, 2, 16)$\\\hline $89$&$(10, 10, 4, 16)$&$90$&$(10, 20, 8, 22)$&$91$&$(14, 14, 8, 
20)$&$92$&$(6, 12, 2, 16)$\\\hline $93$&$(12, 18, 8, 22)$&$94$&$(8, 8, 6, 10)$&$95$&$(16, 16, 14, 
18)$&$96$&$(10, 10, 2, 18)$\\\hline $97$&$(14, 14, 6, 22)$&$98$&$(10, 10, 8, 12)$&$99$&$(14, 14, 12, 16)$&$100$&$(12,12,8,16)$\\\hline $101$&$(12,12,4,20)$&$102$&$(8,16,6,18)$&$103$&$(8,16,4,20)$&$104$&$(8,8,4,12)$\\\hline $105$&$(16,16,12,20)$&$106$&$(8,8,2,14)$&$107$&$(16,16,10,22)$&$108$&$(6,18,4,20)$\\\hline $109$&$(6,6,2,10)$&$110$&$(18,18,14,22)$&&&&\\
\hline
\end{tabular}
$$
\end{itemize}
\end{coro}
\begin{proof}
Applying Theorem~\ref{t1} to the cases $n=5,8,12$, the assertions
follow from a direct computation. Note that in either case the last eleven
entries of the lists above derive from (i)-(xi) of Theorem~\ref{intersectq}. 
\end{proof}

\section{Determination of convex subsets of algebraic Delone
  sets by X-rays}

\begin{defi}\label{xray..}
\begin{itemize}
\item[(a)]
Let $F$ be a finite subset of $\C$, let $u\in
\mathbb{S}^{1}$ be a direction, and let $\mathcal{L}_{u}$ be the set
of lines in the complex plane in direction $u$. Then the {\em
  (discrete parallel)}\/ {\em X-ray} of $F$ {\em in direction} $u$ is
the function $X_{u}F: \mathcal{L}_{u} \rightarrow
\mathbb{N}_{0}:=\mathbb{N} \cup\{0\}$, defined by $$X_{u}F(\ell) :=
\operatorname{card}(F \cap \ell\,)\,.$$
\item[(b)]
Let $\mathcal{F}$ be a collection of finite subsets of
$\C$ and let $U\subset\mathbb{S}^{1}$ be a finite set
of directions. We say that the elements
of $\mathcal{F}$ are {\em determined} by the X-rays in the directions of $U$ if, for all $F,F' \in \mathcal{F}$, one has
$$
(X_{u}F=X_{u}F'\;\,\forall u \in U) \;  \Rightarrow\; F=F'\,.
$$
\end{itemize}
\end{defi}

The following negative result shows that, for algebraic Delone sets $\varLambda$, one has to impose some restriction on the finite subsets of
$\varLambda$ to
be determined. The proof only needs property (Hom).

\begin{fact}\cite[Prop.~ 3.1 and Remark 3.2]{H5}\label{source}
Let $\varLambda$ be an algebraic Delone set and let
$U\subset\mathbb{S}^{1}$ be a finite set of pairwise nonparallel
$\varLambda$-directions. Then the finite subsets of $\varLambda$ are
not determined by the X-rays in the directions of $U$. \qed
\end{fact}

Here, we shall focus on the convex subsets of algebraic Delone
sets. One has the following fundamental 
result which even holds for Delone sets $\varLambda$ with property
(Hom). See Figure~\ref{fig:tilingupolygon} for an
illustration of direction
(i)$\Rightarrow$(ii).

\begin{fact}\cite[Prop.~ 4.6 and Lemma 4.5]{H5}\label{characungen}
Let $\varLambda$ be an algebraic Delone set and let $U\subset\mathbb{S}^{1}$ be a set of two or more pairwise nonparallel $\varLambda$-directions. The following statements are equivalent:
\begin{itemize}
\item[(i)]
The convex subsets of $\varLambda$ are determined by the X-rays in the directions of $U$.
\item[(ii)]
There is no $U$-polygon in $\varLambda$.
\end{itemize}
In addition, if $\operatorname{card}(U)<4$, then there is a $U$-polygon in $\varLambda$.  \qed
\end{fact}

The proof of the following central result uses Darboux's theorem on second 
midpoint polygons; see~\cite{D}, ~\cite{GM} or~\cite[Ch.~1]{G}.

\begin{fact}\cite[Prop.~ 4.2]{GG}\label{uaffine}
Let $U\subset\mathbb{S}^1$ be a finite set of directions. Then there
exists a $U$-polygon if and only if there is an affinely regular
polygon such that each direction in $U$ is parallel to one of its edges.  \qed
\end{fact}

\begin{rem}\label{urem}
Clearly, $U$-polygons have an even number of vertices. Moreover, an
affinely regular polygon with an even number of vertices is a 
$U$-polygon if and only if each direction of $U$ is parallel to one of
its edges. On the other hand, it is important to note that a $U$-polygon need not be affinely regular, even if it is a
$U$-polygon in an algebraic Delone set. For example, there is a
$U$-icosagon in the vertex set of the
T\"ubingen triangle tiling of the plane (a $5$-cyclotomic
model set; see~\cite[Fig.~1,
  Corollary 14 and Example 15]{H4}), which cannot be affinely regular
  since that restricts the number of vertices to $3$, $4$, $5$, $6$ or $10$ by~\cite[Corollary 4.2]{H3}; see also~\cite[Example
  4.3]{GG} for an example in the case of the square lattice. In
  general, there is an affinely regular polygon with $n\geq 3$ vertices in
  an algebraic Delone set $\varLambda$ if and only if
  $\Q(\zeta_n)^+\subset K_{\varLambda}^+$, the latter being a relation
  which (due to property (Alg)) can only hold 
  for finitely many values of $n$; cf.~~\cite[Thm.~ 3.3]{H3}.
\end{rem}

We can now prove our main result on $U$-polygons which is an extension
of~\cite[Thm.~
4.5]{GG}. In fact, we use the same
arguments as introduced by Gardner and Gritzmann in conjunction with
Fact~\ref{crkn4gen} and Theorem~\ref{algcoro}. Note that the result
even holds for arbitrary sets $\varLambda$ with property (Alg).

\begin{theorem}\label{main}
Let $\varLambda$ be an algebraic Delone set. Further, let $U\subset\mathbb{S}^1$ be a set of four or more pairwise
nonparallel $\varLambda$-directions and suppose the existence of a
$U$-polygon. Then the
cross ratio of slopes of any four directions of $U$, arranged in order
of increasing angle with the positive real axis, is an element of the
 set
$\mathcal{C}(K_{\varLambda}^+)$.  Moreover,
$\mathcal{C}(K_{\varLambda}^+)$ is finite and $\operatorname{card}(U)$ is bounded above by a finite number
$b_{\varLambda}\in\N$ that only depends on $\varLambda$. 
\end{theorem}
\begin{proof}
Let $U$ be as in the assertion. By Fact~\ref{uaffine}, $U$ consists
of directions parallel to the edges of an affinely regular
polygon. There is thus a linear automorphism $\Psi$ of
the complex 
plane such that
$$
V:= \big\{ \Psi(u)/| \Psi(u)|\, \big |\, u\in U \big \}
$$ 
is contained in a set of directions that are equally spaced
in $\mathbb{S}^{1}$, i.e., the angle between each pair of adjacent
directions is the same. Since the directions of $U$ are pairwise
nonparallel, we may assume that there is an $m\in\N$ with $m\geq 4$ such that
each direction of $V$ is given by $e^{h\pi i/m}$, where
$h\in\N_{0}$ satisfies $h\leq m-1$. Let $u_{j}$, $1\leq
j\leq 4$, be four
directions of $U$, arranged in order
of increasing angle with the positive real axis. By
Fact~\ref{crkn4gen}, one has
$$q:=\big\langle
\operatorname{sl}(u_{1}),\operatorname{sl}(u_{2}),\operatorname{sl}(u_{3}),\operatorname{sl}(u_{4})\big\rangle\in K_{\varLambda}^+\,.$$ We may assume that $\Psi(u_{j})/| \Psi(u_{j})| =e^{h_{j}\pi i/m}$, where $h_j \in\N_{0}$,
$1\leq j \leq 4$, and, $h_1<h_2<h_3<h_4\leq
m-1$. Fact~\ref{crossratio} now implies
\begin{eqnarray*}
q&=&\big \langle \operatorname{sl}(\Psi(u_1)),\operatorname{sl}(\Psi(u_2)),\operatorname{sl}(\Psi(u_3)),\operatorname{sl}(\Psi(u_4))\big\rangle\\&=&\frac{(\tan (\frac{h_3 \pi}{m})- \tan (\frac{h_1 \pi}{m}))(\tan (\frac{h_4 \pi}{m})- \tan (\frac{h_2 \pi}{m}))}{(\tan (\frac{h_3 \pi}{m})- \tan (\frac{h_2 \pi}{m}))(\tan (\frac{h_4 \pi}{m})- \tan (\frac{h_1 \pi}{m}))}\\
&=&\frac{\sin(\frac{(h_3-h_1)\pi}{m})\sin(\frac{(h_4-h_2)\pi}{m})}{\sin(\frac{(h_3-h_2)\pi}{m})\sin(\frac{(h_4-h_1)\pi}{m})}\,.
\end{eqnarray*}
Setting $k_1:=h_3-h_1$, $k_2:=h_4-h_2$, $k_3:=h_3-h_2$ and
$k_4:=h_4-h_1$, one gets $1\leq k_3<k_1, k_2<k_4\leq m-1$ and
$k_1+k_2=k_3+k_4$. Using
$\sin(\theta)=-e^{-i\theta}(1-e^{2i\theta})/2i$, one finally 
obtains
$$
K_{\varLambda}^+\owns q=\frac{(1-\zeta_{m}^{k_1})(1-\zeta_{m}^{k_2})}{(1-\zeta_{m}^{k_3})(1-\zeta_{m}^{k_4})}=f_m(d)\,,
$$
with $d:=(k_1,k_2,k_3,k_4)$, as in~(\ref{fmd}). Then, $d\in D_m$ if its
first two coordinates are interchanged, if necessary, to ensure that
$k_1\leq k_2$; note that this operation does not change the value of
$f_m(d)$. This proves the first assertion.

Suppose that $\operatorname{card}(U)\geq 7$. Let
$U'$ consist of seven directions of $U$ and let $V':= \{
\Psi(u)/| \Psi(u)|\, |\, u\in U' \}$. We may assume
that all the directions of $V'$ are in the first two quadrants, so one
of these quadrants, say the first, contains at least four directions
of $V'$. Application of the above argument to these four directions gives integers $h_j$ satisfying $0\leq h_1<h_2<h_3<h_4\leq m/2$, where we may also assume, by rotating the directions of $V'$ if
necessary, that $h_1=0$. As above, we obtain a corresponding solution of
$f_m(d)=q\in K_{\varLambda}^+$, where $d\in D_m$. 

By property (Alg) and Theorem~\ref{algcoro}, the set $\mathcal{C}(K_{\varLambda}^+)$
is finite and there is a number
$m_{\varLambda}\in\N$ such that all solutions of $f_{m}(d)\in K_{\varLambda}^+$, where $m\geq 4$ and
$d\in D_{m}$, are either of the form
(xii) or (xiii) of Theorem~\ref{intersectq} or are  given, up to multiplication of $m$
and $d$ by the same factor, by $m=m_{\varLambda}$ and $d$
from a finite list. Without restriction, we may assume that $m_{\varLambda}$ is even. 

Suppose that the above solution is of the form
(xii) or (xiii) of Theorem~\ref{intersectq}. Then using $h_1=0$, one
obtains $h_4=k_4=k+s>m/2$, a contradiction. Thus, our solution derives from
$m=m_{\varLambda}$ and finitely many values of $d\in D_m$. Since
this applies to any four directions of $V'$ lying in the first
quadrant, all such directions correspond to angles with the positive
real axis which are integer multiples of $\pi/m_{\varLambda}$.

We claim that all directions of $V'$ have the latter property. To see
this, suppose that there is a direction $v\in V'$ in the second
quadrant, and consider a set of four directions $v_j$, $1\leq j\leq
4$, in $V'$, where $v_4=v$ and $v_j$, $1\leq j\leq
3$, lie in the first quadrant. Suppose that $v_j=e^{h_j\pi i/m}$, $1\leq
j\leq 4$. Then $h_j$ is an integer multiple of $m/m_{\varLambda}$, for $1\leq j\leq
3$. Again, we obtain a corresponding solution of
$f_m(d)=q\in K_{\varLambda}^+$, where $d\in D_m$. If this solution derives
from the finite list guaranteed by Theorem~\ref{algcoro}, 
then clearly $h_4$ is also an integer multiple of
$m/m_{\varLambda}$. Otherwise, by Theorem~\ref{algcoro}, this
solution is of the form (xii) or (xiii) of Theorem~\ref{intersectq} and we
can take $h_1=0$ as before, whence either $h_2=k$, $h_3=2k$ and
$h_4=k+s$, $1\leq k\leq s/2$, or $h_2=s-k$, $h_3=s$ and
$h_4=k+s$, $s/2\leq k< s$, where $m=2s$. Since
$s=m/2=(m_{\varLambda}/2)(m/m_{\varLambda})$ is an integer 
multiple of $m/m_{\varLambda}$, we conclude in either case that $k$, and hence
$h_4=k+s$, is also an integer multiple of $m/m_{\varLambda}$. This
proves the claim.

It thus remains to examine the case $m=m_{\varLambda}$ in more detail. Let $h_j$,
$1\leq j\leq 4$, correspond to the four directions of $V'$ having the
smallest angles with the positive real axis, so that $h_1=0$ and
$h_j\leq m/2$, $2\leq j\leq 4$. We have already shown that the
corresponding $d=(k_1,k_2,k_3,k_4)$ must occur in the finite list 
guaranteed by Theorem~\ref{algcoro}.  Since $h_j\leq m/2$, $1\leq j\leq 4$, we
also have $k_j\leq m/2$, $1\leq j\leq 4$. This yields only
finitely many quadruples $(h_1,h_2,h_3,h_4)=(0,k_1-k_3,k_1,k_4)$.

Suppose that $h$ corresponds to any other direction of $V'$ and replace $(h_1,h_2,h_3,h_4)$ by
$(h_2,h_3,h_4,h)$. We obtain finitely many
$d=(h_4-h_2,h-h_3,h_4-h_3,h-h_2)\in D_m$, which, by
Theorem~\ref{algcoro},  either occur in (xii) or (xiii) of
Theorem~\ref{intersectq} with $m=m_{\varLambda}$ or occur
in the finite list guaranteed by that result. This gives only finitely many
possible finite sets of more than four directions, which implies that
$\operatorname{card}(U)$ is bounded from above by a finite number that
only depends on $\varLambda$ (since the above analysis only depends on $\varLambda$). 
\end{proof}

Similarly, the next result even holds for arbitrary sets $\varLambda$
with property ($n$-Cyc), where $n\geq 3$.

\begin{theorem}\label{finitesetncr0gen}
Let $n\geq 3$ and let $\varLambda$ be an $n$-cyclotomic Delone set. Further, let $U\subset\mathbb{S}^1$ be a set of four or more pairwise
nonparallel $\varLambda$-directions and suppose the existence of a
$U$-polygon. Then the
cross ratio of slopes of any four directions of $U$, arranged in order
of increasing angle with the positive real axis, is an element of the subset
$\mathcal{C}(K_{\varLambda}^+)$ of\/
$
\mathcal{C}(\Q(\zeta_n)^+)
$.  Moreover
$$\mathcal{C}(\Q(\zeta_n)^+)=\mathcal{C}_{\operatorname{lcm}(2n,12)}(\Q(\zeta_n)^+)$$
is finite 
and $\operatorname{card}(U)$ is bounded above by a finite number
$b_n\in\N$ that only depends on $n$. In particular, one can choose
$b_3=b_4=6$, $b_5=10$, $b_8=8$ and $b_{12}=12$. 
\end{theorem}
\begin{proof}
Employing Theorem~\ref{t1} together with the trivial observation that
$K_{\varLambda}^+\subset\Q(\zeta_n)^+$ for any
$n$-cyclotomic Delone set, the general result follows from the same
arguments as used in the proof of Theorem~\ref{main}. The work of
Gardner and Gritzmann shows that one can choose $b_3=b_4=6$; cf.~\cite[Thm.~ 4.5]{GG}. The specific
bounds $b_n$ for $n=5,8,12$ are obtained by following the proof of
Theorem~\ref{main} and employing Corollary~\ref{coro8125}.

More precisely, let $n=8$ (whence $\operatorname{lcm}(2n,12)=48$) and suppose that $\operatorname{card}(U)\geq 7$. Let
$U'$ consist of seven directions of $U$ and let $V':= \{
\Psi(u)/| \Psi(u)|\, |\, u\in U' \}$, with $\Psi$ as described in the
proof of Theorem~\ref{main}. Then all
directions of $V'$ correspond to angles with the positive
real axis which are integer multiples of $\pi/48$ and it suffices to examine the case $m=48$ in more detail. Let $h_j$,
$1\leq j\leq 4$, correspond to the four directions of $V'$ having the
smallest angles with the positive real axis, so that $h_1=0$ and
$h_j\leq m/2=24$, $2\leq j\leq 4$. The corresponding $d=(k_1,k_2,k_3,k_4)$ must occur in (1)-(82) of
Corollary~\ref{coro8125}(b). Since $h_j\leq 24$, $1\leq j\leq 4$, we
also have $k_j\leq 24$, $1\leq j\leq 4$. The only possibilities are (1), (3), (15), (19), (27), (28), (33), (34),
(47), (67), (76) and (81) of
Corollary~\ref{coro8125}(b). These yield
\begin{eqnarray*}
(h_1,h_2,h_3,h_4)&\in&\big\{(0,2,6,20), (0,6,12,18), (0,2,4,12),
(0,6,12,24),\\&& \hphantom{\big\{} (0,2,10,24),(0,6,18,24), (0,6,8,14),
(0,6,9,21), \\&& \hphantom{\big\{}(0,8,10,20), (0,6,16,22), (0,8,16,24), (0,8,12,20)\big\}\,.
\end{eqnarray*}
Suppose that $h$ corresponds to any other direction of $V'$ and replace $(h_1,h_2,h_3,h_4)$ by
$(h_2,h_3,h_4,h)$. The corresponding $d$ either occur in (xii) or (xiii) of
Theorem~\ref{intersectq} with $m=48$ or occur in (1)-(82) of
Corollary~\ref{coro8125}(b).  We obtain $(18,h-6,14,h-2)$, $(12,h-12,6,h-6)$,
$(10,h-4,8,h-2)$, $(18,h-12,12,h-6)$, $(22,h-10,14,h-2)$,
$(18,h-18,6,h-6)$, $(8,h-8,6,h-6)$, $(15,h-9,12,h-6)$,
$(12,h-10,10,h-8)$, $(16,h-16,6,h-6)$, $(16,h-16,8,h-8)$ and
$(12,h-12,8,h-8)$. The only possibilities are $h=24,30,36,42$ for $(12,h-12,6,h-6)$, $h=26,40$ for $(10,h-4,8,h-2)$,
$h=30,36,42$ for $(18,h-12,12,h-6)$, $h=38,46$ for $(22,h-10,14,h-2)$,
$h=36,42$ for $(18,h-18,6,h-6)$, $h=34$ for $(12,h-10,10,h-8)$,
$h=32,40$ for $(16,h-16,8,h-8)$ and $h=34$ for $(12,h-12,8,h-8)$. It
follows that the only possible sets of more than four directions only comprise
directions of the form $e^{h\pi i/48}$ and 
are given by  the ranges 
\begin{eqnarray*}&&\{0,8,16,24,32,40\},
 \{0,8,12,20,34\},
\{0,6,12,18,24,30,36,42\},\\&& \{0,2,4,12,26,40\},
\{0,6,12,24,30,36,42\}, \{0,2,10,24,38,46\},\\&& \{0,6,18,24,36,42\}, \{0,8,10,20,34\} 
\end{eqnarray*}
of $h$. In particular, $\operatorname{card}(U)\leq 8$.  

With the help of
Corollary~\ref{coro8125},
the cases $n=5,12$ can be treated analogously with the following
results. 

For $n=12$, the only possible sets of more than four directions only comprise
directions of the form $e^{h\pi i/24}$ and are given by the ranges 
\begin{eqnarray*}&&\{0,4,8,12,16,18,20,22\},  \{0,4,6,10,14,16,18,20,22\},\\&& \{0,2,4,10,12,14,18,20,22\},
 \{0,2,4,8,12,14,16,18,20,22\},\\&& \{0,2,4,6,8,10,12,14,16,18,20,22\},
  \{0,2,6,12,16,18,20,22\},\\&& \{0,2,4,12,14,20,22\},  \{0,4,6,12,14,16,18,20,22\},
\{0,2,8,12,18,20,22\},\\&&   \{0,2,6,10,14,16,18,20,22\},\{0,2,8,10,16,18,20,22\},
\{0,2,10,12,20,22\}
\end{eqnarray*} 
of $h$, whence $\operatorname{card}(U)\leq 12$. 

For $n=5$, the only possible sets of more than four directions only comprise
directions of the form $e^{h\pi i/60}$ and are given by the ranges 
\begin{eqnarray*}&&
\{0,10,20,30,40,50\},\{0,6,24,30,48,54\},
\{0,2,4,10,32,54\},\\&& \{0,4,8,18,34,50\},
\{0,6,10,16,38\},\{0,6,12,24,30,36,42,48,54\},\\&& \{0,10,18,28,44\},
\{0,12,18,30,36,42,48,54\},\{0,6,8,16,38\},\\&& \{0,10,14,28,34\},\{0,2,8,30,52,58\},
\{0,4,14,30,46,56\},\\&& \{0,6,18,30,42,48,54\},\{0,6,12,30,36,42,48,54\},\\&& 
\{0,6,12,18,24,30,36,42,48,54\}, \{0,6,18,24,36,42,48,54\}
\end{eqnarray*}
of $h$, whence $\operatorname{card}(U)\leq 10$ in this case. 
\end{proof}

Without further mention, the following result will be used in Remark~\ref{maxrem} below.

\begin{lem}\label{upollem}
Let $\varLambda$ be a $K$-algebraic model set and let
$U\subset\mathbb{S}^1$ be a finite set of directions. The following
statements are equivalent:
\begin{itemize}
\item[(i)]
There is a $U$-polygon in $\varLambda$.
\item[(ii)]
For any $K$-algebraic model set $\varLambda'$, there is a $U$-polygon
in $\varLambda'$. 
\end{itemize}
\end{lem}
\begin{proof}
The assertion follows immediately from Proposition~\ref{cmsads} together with~\cite[Fact 4.4]{H5}. 
\end{proof}

\begin{figure}
\centerline{\epsfysize=0.6\textwidth\epsfbox{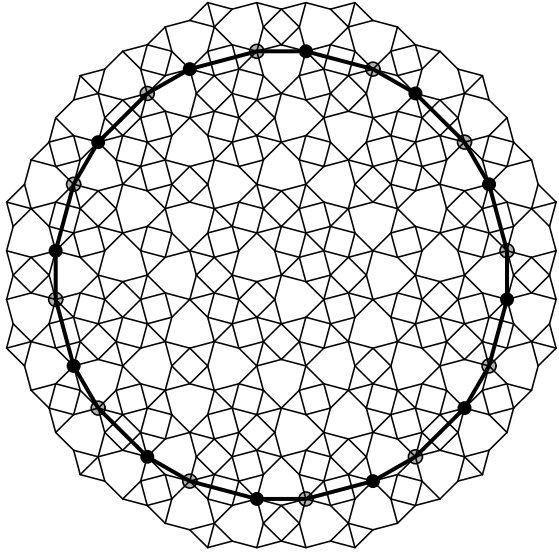}}
\caption{The boundary of a $U$-polygon in
  the vertex set $\varLambda$ of the twelvefold 
  shield tiling, where $U$ is the set of
  twelve pairwise nonparallel $\varLambda$-directions
 given by the edges and diagonals of the central regular
 dodecagon. The vertices of $\varLambda$ in the interior of
 the  $U$-polygon together with the vertices indicated by the black and grey dots, respectively,
give two different convex subsets of $\varLambda$ with the same X-rays in the
directions of $U$.}
\label{fig:tilingupolygon}
\end{figure}

\begin{rem}\label{maxrem}
 The work of
Gardner and Gritzmann shows that $b_3=b_4=6$ is best possible for any
$3$- or $4$-cyclotomic model set; cf.~\cite[Example 4.3]{GG}. The $U$-icosagon in the vertex set of the T\"ubingen
 triangle tiling from Remark~\ref{urem} has the property that
 $\operatorname{card}(U)=10$; see~\cite[Figure 1]{H4}. This  shows that, for any $5$-cyclotomic model set, the number $b_5=10$ is best
 possible. Fig.~\ref{fig:tilingupolygon} shows a $U$-polygon with $24$ vertices
 in the vertex set of the
 shield tiling with $\operatorname{card}(U)=12$, wherefore $b_{12}=12$
 is best possible for any $12$-cyclotomic model set. A similar example
 of a $U$-polygon with $16$ vertices  in the vertex set of the
 Ammann-Beenker tiling with $\operatorname{card}(U)=8$ shows that $b_{8}=8$ is
 best possible for any $8$-cyclotomic model set; cf.~\cite[Fig.~2]{H5}. $U$-polygons
 of {\em class $c\geq 4$} (i.e., $U$-polygons with $4$ {\em
   consecutive edges} parallel to directions of $U$) in cyclotomic
 model sets were studied in~\cite{H4}. By~\cite[Corollary 14]{H4} (see
 also~\cite[Thm.~ 12]{DP}), the
 existence of a 
 $U$-polygon of class $c\geq 4$ in an $n$-cyclotomic model set with
 $n\not\equiv 2\pmod 4$ having the property that $\phi(n)/2$ is equal to one or a prime number
 implies that $\operatorname{card}(U)\leq a_n$, where $a_3=a_4=6$, $a_8=8$,
 $a_{12}=12$ and $a_n=2n$ for all other such values of $n$. In particular,
 one observes the coincidence $b_n=a_n$ for $n=3,4,5,8,12$; cf.~Theorem~\ref{finitesetncr0gen}. However, there does not seem to be a reason why the least
 possible numbers $b_n$ in Theorem~\ref{finitesetncr0gen} may not be
 larger than $a_n$ for other $n\geq 3$ having the above property.         
\end{rem}

Summing up, we finally obtain our main result on the
determination of convex subsets of algebraic Delone sets;
see~\cite[Thm.~ 4.21]{H5} for a
weaker version. 

\begin{theorem}\label{dtmain}
Let $\varLambda$ be an algebraic Delone set.
\begin{itemize}
\item[(a)]
There are sets of four pairwise nonparallel $\varLambda$-directions
such that the convex subsets of $\varLambda$ are determined by the corresponding
X-rays. In addition, less than four pairwise nonparallel $\varLambda$-directions never
suffice for this purpose.
\item[(b)]
There is a finite number $c_{\varLambda}\in\N$ such that the convex subsets of
$\varLambda$ are determined by the X-rays in any set of $c_{\varLambda}$
pairwise nonparallel $\varLambda$-directions. 
\end{itemize}
\end{theorem}
\begin{proof}
To prove~(a), it suffices by Fact~\ref{characungen} and
Theorem~\ref{main} to take any set of four pairwise
nonparallel $\varLambda$-directions such that the cross ratio of their
slopes, arranged in order
of increasing angle with the positive real axis, is not an element of
the finite set $\mathcal{C}(K_{\varLambda}^+)$. Since
$\varLambda$ is relatively dense, the set of $\varLambda$-directions
is dense in $\mathbb{S}^1$. In particular, this shows that the set of
slopes of $\varLambda$-directions is infinite. For example by fixing three
pairwise nonparallel 
$\varLambda$-directions and letting the fourth one vary, one sees from
this that
the set of cross ratios of
slopes of four pairwise
nonparallel $\varLambda$-directions, arranged in order
of increasing angle with the positive real axis, is infinite as well.
The assertion follows. The additional statement follows immediately
from Fact~\ref{characungen}. Part~(b) is a direct consequence of
Fact~\ref{characungen} and Theorem~\ref{main}.
\end{proof}

The following result improves~\cite[Thm.~
4.33]{H5} and particularly solves Problem 4.34
of~\cite{H5}; cf.~Example~\ref{algex} and compare~\cite[Thm.~ 5.7]{GG}.

\begin{theorem}\label{dtmain2}
Let $n\geq 3$ and let $\varLambda$ be an $n$-cyclotomic Delone set.
\begin{itemize}
\item[(a)]
There are sets of four pairwise nonparallel $\varLambda$-directions
such that the convex subsets of $\varLambda$ are determined by the corresponding
X-rays. In addition, less than four pairwise nonparallel $\varLambda$-directions never
suffice for this purpose.
\item[(b)]
There is a finite number $c_n\in\N$ that only depends on $n$ such that the convex subsets of
$\varLambda$ are determined by the X-rays in any set of $c_n$
pairwise nonparallel $\varLambda$-directions. In particular, one can
choose $c_3=c_4=7$, $c_5=11$, $c_8=9$ and $c_{12}=13$. 
\end{itemize}
\end{theorem}
\begin{proof}
Part~(a) follows immediately from Theorem~\ref{dtmain}(a). Note that, by Fact~\ref{characungen} and
Theorem~\ref{finitesetncr0gen}, it suffices to take any set of four pairwise
nonparallel $\varLambda$-directions such that the cross ratio of their
slopes, arranged in order
of increasing angle with the positive real axis, is not an element of
the finite set $\mathcal{C}(\Q(\zeta_n)^+)$. Part~(b) is a direct consequence of
Fact~\ref{characungen} in conjunction with 
Theorem~\ref{finitesetncr0gen}.
\end{proof}

\begin{rem}\label{rembest}
Remark~\ref{maxrem} shows that, for any $n$-cyclotomic
model set with $n=3,4,5,8,12$, the
number $c_n$ above is best possible with respect to the numbers of X-rays used. As already explained in the introduction, for practical applications, one additionally has to make sure that the
$\varLambda$-directions used yield densely occupied lines in
$\varLambda$. For the practically most relevant case of $n$-cyclotomic
model sets with $n=3,4,5,8,12$, this
can actually be achieved; cf.~\cite[Remark 5.8]{GG} and~\cite[Sec.~4]{H5} for examples of suitable sets of four pairwise nonparallel
$\varLambda$-directions in these cases. For the latter examples also recall that, for any $n$-cyclotomic model set $\varLambda$, the set of
$\varLambda$-directions is precisely the set
of $\Z[\zeta_n]$-directions; cf. Remark~\ref{okdirections} and
Example~\ref{algex}. It was shown in~\cite[Prop.~ 3.11]{H2} that {\em icosahedral model sets}
$\varLambda\subset\R^3$ can be sliced orthogonal to a fivefold axis of
their underlying $\Z$-module 
into $5$-cyclotomic model sets. Applying Theorem~\ref{dtmain2} to
each such slice, one sees that the convex subsets of $\varLambda$
are determined by the X-rays in suitable four and
any eleven pairwise
nonparallel $\varLambda$-directions orthogonal to
the slicing axis.    
\end{rem}

\section{Determination of convex bodies by continuous X-rays}

In~\cite{GM}, the following continuous version of
Fact~\ref{characungen} was
shown;  compare Fact~\ref{uaffine}. Here, the {\em continuous X-ray} of a
{\em convex body\/} $K\subset\C$ (i.e., $K$ is compact with nonempty interior) in direction $u\in\mathbb{S}^1$ gives the length of
each chord of $K$ parallel to $u$ and the concept of determination is
defined as in the discrete case; cf.~ \cite{G}, \cite{GM} for details. 

\begin{fact}\label{characuncont}
Let $U\subset\mathbb{S}^{1}$ be a set of two or more pairwise nonparallel directions. The following statements are equivalent:
\begin{itemize}
\item[(i)]
The convex bodies in $\C$ are determined by the continuous X-rays in the directions of $U$.
\item[(ii)]
There is no $U$-polygon.
\end{itemize}
In addition, if $\operatorname{card}(U)<4$, then there is a $U$-polygon.  \qed
\end{fact}

Employing Fact~\ref{characuncont} instead of Fact~\ref{characungen},
the following result follows from the same arguments as used in the proofs
of  Theorems~\ref{dtmain} and~\ref{dtmain2};
compare~\cite[Thm. 6.2]{GG}. Note that neither the uniform
discreteness of $\varLambda$ nor property (Hom) 
are needed in the proof. More precisely, our proof of part~(a) needs property (Alg)
and the relative denseness of $\varLambda$, whereas part~(b) and the
additional statement 
hold for arbitrary sets $\varLambda$ with property (Alg) and
($n$-Cyc) (where $n\geq 3$), respectively. 

\begin{theorem}\label{tmain}
Let $\varLambda$ be an algebraic Delone set.
\begin{itemize}
\item[(a)]
There are sets of four pairwise nonparallel $\varLambda$-directions
such that the convex bodies in $\C$ are determined by the
corresponding continuous 
X-rays. In addition, less than four pairwise nonparallel $\varLambda$-directions never
suffice for this purpose.
\item[(b)]
There is a finite number $c_{\varLambda}\in\N$ such that the convex
bodies in $\C$ are determined by the continuous X-rays in any set of $c_{\varLambda}$
pairwise nonparallel $\varLambda$-directions. 
\end{itemize}
Moreover, for any
$n$-cyclotomic Delone set $\varLambda$, there is a finite number
$c_n\in\N$ that only depends on $n$ such that the convex
bodies in $\C$ are determined by the continuous X-rays in any set of $c_n$
pairwise nonparallel $\varLambda$-directions. In particular, one can choose $c_3=c_4=7$, $c_5=11$, $c_8=9$ and $c_{12}=13$.\qed
\end{theorem}

\begin{rem}
Employing the $U$-polygons from
Remark~\ref{maxrem}, it is straightforward to show
that the above numbers $c_n$, where $n=3,4,5,8,12$, are best
possible. 
\end{rem}

\section*{Acknowledgements}
This work was supported by the German Research Council (Deutsche
Forschungsgemeinschaft), within the CRC 701. C.~H.~is
grateful to Richard J.~Gardner for his cooperation and
encouragement. The authors thank M.~Baake for 
useful comments on the manuscript.

\end{document}